\documentclass{article}

\usepackage{arxiv}

\usepackage[utf8]{inputenc} 
\usepackage[T1]{fontenc}    
\usepackage{hyperref}       
\usepackage{url}            
\usepackage{booktabs}       
\usepackage{amsfonts}       
\usepackage{nicefrac}       
\usepackage{microtype}      
\usepackage{lipsum}
\usepackage{graphicx}
\graphicspath{ {./images/} }
\usepackage{amsthm, amsfonts, amsmath, amssymb}
\usepackage{multicol}
\usepackage[all,2cell]{xy} \UseAllTwocells \SilentMatrices

\newtheorem{theorem}{Theorem}[section]
\newtheorem{proposition}[theorem]{Proposition}
\newtheorem{corollary}[theorem]{Corollary}
\newtheorem{lemma}[theorem]{Lemma}

\newtheorem{definition}[theorem]{Definition}
\newtheorem{example}[theorem]{Example}
\newtheorem{remark}[theorem]{Remark}

\usepackage[overload]{empheq}
\usepackage{nccmath}

\newcommand{\Q}{\mathcal{Q}}
\newcommand{\HH}{\mathcal{H}}

\usepackage{quiver}
\usepackage{tikz-cd}
\usepackage{ dsfont }
\usepackage{mathtools}

\usepackage{tikz}
\usetikzlibrary{matrix}
\usetikzlibrary{calc}
\usepackage{blindtext}
\usepackage{authblk} 

\title{On sheaves on semicartesian quantales and their truth values}

\author[1]{Ana Luiza Tenório}
\author[2]{Caio de Andrade Mendes}
\author[3]{Hugo Luiz Mariano}
\affil[1]{Department of Mathematics, University of São Paulo, \texttt{analuiza@ime.usp.br}}
\affil[2]{Department of Mathematics, University of São Paulo, \texttt{caio.mendes@usp.br}}
\affil[3]{Department of Mathematics, University of São Paulo, \texttt{hugomar@ime.usp.br}}

\begin{document}
\maketitle
\begin{abstract}{
In this paper, introduce a new definition of sheaves on semicartesian quantales, providing first examples and categorical properties. We note that our sheaves are similar to the standard definition of a sheaf on a locale, however, we prove in that in general it is not an elementary topos - since the lattice of external truth values of $Sh(Q)$, $Sub(1)$,  is canonically isomorphic to the quantale $Q$ - placing this paper as part of a greater project towards a monoidal (not necessarily cartesian) closed version of elementary topos. To start the study the logical aspects of the category of sheaves we are introducing, we explore the nature of the ``internal truth value objects" in such sheaves categories. More precisely, we analyze two candidates for subobject classifier for different subclasses of commutative and semicartesian quantales.   }

\end{abstract}


\section{Introduction}
Sheaf Theory is a well-established area of research with applications in Algebraic Topology \cite{dimca2004sheaves}, Algebraic Geometry \cite{grothendieck1972topos}, Geometry \cite{kashiwara2013sheaves}, Logic \cite{maclane1992sheaves}, and others. A sheaf on a locale $L$ is a functor $F: L^{op} \rightarrow Set$ that satisfies certain gluing properties expressed by an equalizer diagram. However, for quantales -- a non-idempotent and non-commutative generalization of locales introduced by C.J. Mulvey \cite{mulvey86quantales} -- there are many definitions of sheaves on quantales: in  \cite{borceux1986quantales}, sheaves on  quantales are defined with the goal of forming Grothendieck toposes from quantales. In \cite{miraglia1998sheaves}, the sheaf definition  preserves an intimate relation with $Q$-sets, an object introduced in the paper as a proposal to generalize $\Omega$-sets, defined in \cite{fourman1979sheaves}, for $\Omega$ a complete Heyting algebra\footnote{Given a proper notion of morphisms of $\Omega$-sets, the category of $\Omega$-sets is equivalent to the category of sheaves on $\Omega$.}. More recently, in \cite{aguilar2008sheaves}, sheaves are functors that make a certain diagram an equalizer. Besides, an extensive work about sheaves on \textit{involutive quantale}, which goes back to ideas of Bob Walters \cite{walters1981sheaves}, were recently studied by Hans Heymans, Isar Stubbe \cite{heymans2012grothendieck}, and Pedro Resende \cite{resende2012groupoid}, for instance.

In this paper, we provide the basic definitions and properties of quantales in Section \ref{sec:quantales}, introducing the notions of Artinian and geometric quantales with examples that use von Neumann regular rings and the extended half-line $[0,+\infty]$, respectively.
In Section \ref{sec:sheaves_quantales}, we provide a new definition of sheaves on semicartesian quantales (as a functor that forms an equalizer diagram), given first examples and categorical properties. We note that our sheaves are similar to the standard definition of a sheaf, however, we prove in Section \ref{sec:comparing_with_Groth_topos}, in general, it is not an elementary topos - since the lattice of external truth values of $Sh(Q)$, $Sub(1)$,  is canonically isomorphic to the quantale $Q$ - , placing this paper as part of a greater project towards a monoidal closed but non-cartesian closed version of elementary topos. To explore our interest in the logical aspects of the category of sheaves we are introducing, we provide in Section \ref{sec:truthvalues} a detailed construction and analysis of two candidates for subobject classifier for different subclasses of commutative and semicartesian quantales:  In other words, we investigate the nature of the ``internal truth value objects" in such sheaves categories.

\textit{We assume that the reader is comfortable with some concepts of category theory such as functors, natural transformations, (co)limits constructions, and adjointness.}

\section{Quantales}\label{sec:quantales}

\begin{definition}
A \textbf{quantale} is a structure $Q = (Q, \leq, \odot)$ where: $(Q, \leq)$ is a complete lattice; $(Q, \odot)$ is a  semigroup -  i.e. the  binary operation $\odot: Q \times Q \to Q$ (called multiplication) is associative; moreover $Q = (Q, \leq, \odot)$  satisfies the following distributive laws: for all $a\in  Q$ and $\{b_i\}_{i\in I} \subseteq Q$
\begin{align*}
    a \odot (\bigvee\limits_{i \in I}b_i) = \bigvee\limits_{i \in I}(a \odot b_i) \enspace \mbox{ and } \enspace
    & (\bigvee\limits_{i \in I}b_i)  \odot a = \bigvee\limits_{i \in I}(b_i \odot a)
\end{align*}
\end{definition}

\begin{remark} Note that:
\begin{enumerate}
    \item In any quantale $Q$ the multiplication is increasing in both entries, where increasing in the second entry means that given $a, b, b' \in Q$ such that $b \leq b'$, we have $a \odot b \leq a \odot b'$. Indeed $ a \odot b' = a \odot (b \vee b') = (a \odot b) \vee (a \odot b')$, thus $a \odot b \leq a \odot b'$.
    \item Since  the least element of  the quantale $Q$, here denoted by $0$ (or $\bot$), is also the supremum of the emptyset, note that  $a \odot 0 = 0 = 0 \odot a, \forall a \in Q$.
\end{enumerate}
\end{remark}\label{remark_quantale_properties}

Similarly, a {\em unital} quantale is a structure $(Q, \leq, \odot, 1)$, where $(Q, \odot, 1)$ is a {\em monoid}. Note that the associativity of $\odot$ and the identity element provide a (strict) monoidal structure to $(Q, \leq, \odot, 1)$, viewed as a poset category.

If $(Q, \leq)$ is a complete lattice and it holds the distributive laws above taking the infimum as the associative binary operation ($\odot = \wedge$) turns  $\odot$ into a commutative operation where $\top$, the largest member of $Q$, is its neutral element. In such a case, we obtain a \textbf{locale}. Thus every locale is a unital quantale. The main example of locale used in sheaf theory is the locale $\mathcal{O}(X)$ of open subsets of a topological space $X$, where the order relation is given by the inclusion, the supremum is the union, and the finitary infimum is the intersection. We list below some examples of unital quantales that are not locales:

\begin{example}\label{ex_quantales}
\begin{enumerate}
    \item The extended half-line $[0,\infty]$ with order $\geq$, and the usual sum of real numbers as the multiplication. Since the order relation is $\geq$, the top element is $0$ and the bottom elements is $\infty$; 
    
    \item The extended natural numbers $\mathbb{N}\cup \{\infty\}$, with the same quantalic structure of $[0,\infty]$;
    
    \item The set $\mathcal{I}(R)$ of ideals of a commutative and unital ring $R$ with order $\subseteq$, the inclusion of ideals, and the multiplication as the multiplication of ideals. The supremum is the sum of ideals, the top element is $R$ and the trivial ideal is the bottom;

    \item The set of closed right (or left) ideals of a unital $C^*$-algebra, the order is the inclusion of closed right (or left) ideals, and the multiplication is the topological closure of the multiplication of the ideals.
\end{enumerate}
\end{example}

For more details and examples we recommend \cite{rosenthal_quantales}. 

\begin{definition}
A quantale $Q = (Q, \leq, \odot)$ is
\begin{enumerate}
    \item \textbf{commutative} when $(Q, \odot)$ is a commutative semigroup;
    \item  \textbf{idempotent} when $a \odot a = a$, for $a \in Q$;
    \item \textbf{right-sided} when $a \odot \top = a$, for all $a \in Q$, where $\top$ is the top member of the poset;
    \item \textbf{semicartesian}  when $a \odot b \leq a, b$, for all $a,b \in Q$;
    \item \textbf{integral} when $Q$ is unital and $1 = \top$;
    \item \textbf{divisible} when $ a \leq b \implies \exists x, b \odot x = a,$; 
    \item \textbf{strict linear} when $\leq$ is a linear order and $a \neq 0, b \neq c \implies a \odot b \neq a \odot c$, for all $a, b, c \in Q.$
\end{enumerate}
\end{definition}

The quantales $[0,\infty]$, $\mathbb{N} \cup \{\infty\}$, and $\mathcal{I}(R)$ are commutative and integral unital quantales. Besides, they are also divisible (for $\mathcal{I}(R)$, we have to take $R$ a PID). The last example is neither commutative nor semicartesian but it is right-sided (resp. left-sided) quantale. \cite{rosenthal_quantales}.

\begin{remark} \label{quantale-re} Note that:
\begin{enumerate}
\item A quantale $(Q, \leq, \odot)$ is semicartesian iff $a\odot b \leq a \wedge b$, for all $a, b \in Q$.
\item Let $Q$ be a unital quantale, then it is integral iff it  is semicartesian.  Indeed: suppose that $Q$ is integral, since $b \leq \top$ we have $a \odot b \leq a \odot \top = a \odot 1 = a$, then $Q$ is semicartesian; conversely, suppose that $Q$ is semicartesian, since $\top = \top \odot 1 \leq 1$, then $\top = 1$.

    \item Let $Q$ be a integral commutative quantale. Define a binary relation on $Q$ by:  
$x \preceq y \Leftrightarrow x = x  \odot y$. Then: (i) $\preceq$ is a partial order; (ii)  $x \preceq y \Rightarrow x \leq y$;
    (iii) $x \preceq y, y \leq z \Rightarrow x \preceq z$;  (iv)  $e \in Idem(Q) \Rightarrow $  ($e \preceq x \Leftrightarrow e \leq x$). 

\end{enumerate}
\end{remark}\label{remark_quantale_semi_properties}

The following result explains why we claim that our sheaves are orthogonal to notions of sheaves on idempotent quantales: 

\begin{proposition}\label{prop:when_quantale_is_locales}
If $Q$ is a semicartesian quantale, then by adding the idempotency we obtain that  $\odot = \wedge$. In other words, $Q$ is a locale.
\end{proposition}
\begin{proof}
Since $Q$ is semicartesian, by Remark \ref{remark_quantale_semi_properties}, for any $b, c$ in $Q$, $b \odot c \leq b,c$.
If $Q$ is idempotent, for any $a \in Q$ such that $a \leq b$ and $a \leq c$ we have $a = a\odot a \leq b \odot c$, because the multiplication is increasing in both entries. 
So if $Q$ is semicartesian and idempotent, by the transitivity of the order relation:
$$a\leq b \odot c \iff a\leq b \mbox{ and } a\leq c $$
Thus the multiplication satisfies the definition of the meet operation.
\end{proof}
The above proposition is just a particular case of \cite[Proposition 2.1]{nlab:quantale}. Observe that any notion of sheaf on a idempotent and semicartesian quantale is necessarily a sheaf on a locale, which is a well-established object of study. 

{\bf Construction of quantales:}

\begin{enumerate}
    
\item  Notice that given a family of quantales $\{Q_i:  i \in I\}$ the cartesian product  $\prod_{i \in I}Q_i$ with component-wise order (i.e., 
$(a_i)_i \leq (b_i)_i \iff a_i \leq b_i, \forall i \in I$) is a quantale. 
Define $\bigvee_{j \in I} (a_{ij})_i = (\bigvee_{j \in I}a_{ij})_i$, $\bigwedge_{j \in I} (a_{ij})_i = (\bigwedge_{j \in I}a_{ij})_i$ and $(a_{i})_i \odot (b_{i})_i = (a_{i} \odot b_{i})_i$. All verifications are straightforward but we will check one of the distributive laws:

\begin{align*}
  \bigvee_{j \in I}(a_i \odot b_{ij})_i = (\bigvee_{j \in I}a_i \odot b_{ij})_i = (a_i \odot \bigvee_{j \in I} b_{ij})_i  = (a_i)_i \odot \bigvee_{j \in I}(b_{ij})_i 
\end{align*}

It is easy to see that $\prod_{i \in I}Q_i$ is a semicartesian/commutative quantale whenever each $Q_i$ is a  semicartesian/commutative quantale.

\item If $Q$ is a commutative semicartesian quantale, it is straitforward to check that given  $l\in Idem(Q)$ and $u \in Q $ such that $l \leq u$, then the subset $[l,u] = \{x \in Q: l \leq x\leq u\}$ is closed under $\bigvee$ and $\odot$, thus it determines a ``interval subquantal" that is also semicartesian and commutative. 
\end{enumerate}

An example of a semicartesian quantale that is not integral is constructed as follow: let $Q$ be a integral and not idempotent quantale. Given $a \in Q \setminus Idem(Q)$,  then the interval $[\bot, a]$ is a non unital semicartesian quantale.

\begin{remark}\label{remark:def_implication}
Every commutative unital quantale $Q$ can be associated to a closed monoidal symmetric\footnote{A monoidal category is category equipped with a tensor functor and satisfies coherence laws given by commutative diagrams; the symmetry establishes a kind of commutativity for the tensor; and closed means that there is an isomorphism $Hom(a \otimes c, b) \cong Hom(c, b^a)$, and under such bijection the arrow $ev \in Hom(a \otimes b^a, b)$ that corresponds to $id \in Hom(b^a, b^a)$.} poset category $Q$, where exists a unique arrow in $Hom(a,b)$ iff $a\leq b$. Note that the product $\prod$, coproduct $\coprod$, and tensor $\otimes$ are defined, respectively, by the infimum, $\bigwedge$, the supreme $\bigvee$, and the dot $\odot$ and the ``exponential'' is given by $b^a = \bigvee\{c \in Q : a \odot c \leq b\}$, where $b^a$ is an alternative notation for $a \rightarrow b$.
\end{remark}

Now, we introduce an operation that sends elements of a commutative and semicartesian quantale $Q$ into a idempotent element in the locale $Idem(Q)$. In those conditions we define
   $$ q^{-}   := \bigvee \{p \in Idem(Q) : p\leq q \odot p\} $$
Since $Q$ is semicartesian and commutative, note that $p \leq q \odot p$ iff $p = q \odot p = p \odot q$.

Now, we list properties of $(-)^-: Q \to Idem(Q)$. 

\begin{proposition}\label{prop:properties_of_idempotentaprox}
If $Q$ is a commutative and semicartesian quantale, and $\{q_i : i \in I\} \subseteq Q$, then
\begin{multicols}{2}
\begin{enumerate}
    \item ${0}^- = 0$ and ${1}^- =1$ (if $Q$ is unital)
    \item $q^- \leq q$ 
    \item $q^- \odot q = q^-$ 
    \item $q = q^- \Leftrightarrow q\odot q = q$ 
    \item $q^- \odot q^- = q^-$ 
    \item $q^- = max\{ e \in Idem(Q) : e \leq q \}$  
    \item $q^{-{^-}} = q^-$
    \item  $p \leq q$ and $x \odot p = x $, then $x \odot q = x$
    \item $p \leq q$ $\Rightarrow$ $p^- \leq q^-$  
     $\Leftrightarrow$ $p^- \odot q^- = p^-$ 
    \item $(a\odot b)^- = a^-\odot b^-$
    \item ${q_j}^- \odot \bigvee_{i \in I} q_i = {q_j}^-$ 
    \item ${\bigvee_{i} q_i^-} \leq (\bigvee_{i} {q_i})^-$
\end{enumerate}
\end{multicols}
\end{proposition}
\begin{proof}
\begin{enumerate}
    \item Straightforward.
    
    \item If $e \in Idem(\Q)$ is such that $e \leq q \odot e$, then $e \leq q \odot 1  = q$. Thus, by the definition of $q^-$ as a least upper bound,  $q^- \leq q$.
    
    \item Since multiplication distributes over arbitrary joins,
    $$q^- \odot q = \bigvee\{p \odot q \,|\, p = q\odot p, p \in Idem(Q)\} = \bigvee\{p \in Idem(Q)\,|\, p = q\odot p\} = q^-$$
    
    \item $(\Rightarrow)$ From the previous item.
    
    $(\Leftarrow)$ By maximality, $q \leq q^-$, thus the result follows from item $(2)$.
    
    \item Since multiplication distributes  over arbitrary joins,
     \begin{align*}
         q^- = q^- \odot 1 \geq q^- \odot q^- &= \bigvee\{p \odot q^-  Idem(Q)\,|\, p \in  Idem(Q), p = q\odot p\}\\
         &= \bigvee\{p \odot p' \,|\, p, p' \in Idem(Q) , p = q\odot p, p' = q \odot p' \}\\
         &\geq \bigvee\{p \odot p \,|\, p \in Idem(Q) , p = q\odot p  \} \\
         &=  \bigvee\{p \,|\, p \in Idem(Q) , p = q\odot p  \} \\
         &=  q^-
     \end{align*}

    \item By items $(2)$ and  $(5)$, $q^- \in \{e \in Idem(Q): e \leq q\}$. If $e \in Idem(Q)$ is such that $e \leq q$, then $ e = e \odot e \leq q \odot e  \leq 1 \odot e = e$, thus $e = e \odot q$; then $e \leq q^-$, by the definition of $q^-$ as a l.u.b.

    \item By item $(2)$, $q{^{-^{-}}} \leq q^-$. On the other hand, by items $(4)$ and $(5)$ and maximality of ${q^-}^-$, we have $q^- \leq q{^{-^{-}}} $. 
 
  \item Since $ x = x \odot p \leq x \odot q \leq x \odot 1 = x$.
    
    \item    Suppose $p \leq q$. Then by items $(3)$ and $(8)$, $p^- \odot q = p^-$. 
    
    By item $(5)$, $p^- \in Idem(Q)$ and, by maximality of $q^-$, $p^- \leq q^-$.
    
    Since $p^-, q^- \in Idem(Q)$ (item $(5)$) then, by the argument in the proof of item $(6)$, we have $p^- \leq q^-$ iff  $p^- \odot q^- = p^-$.

\item Note that $a^- \odot b^-$ is an idempotent such that $a^- \odot b^- \odot a \odot b = a^- \odot b^-$. 
    So $(a \odot b)^- \geq a^-\odot b^-$.

On the other hand, by item $(9)$, $(a\odot b)^- \odot a^- = (a\odot b)^- = (a\odot b)^- \odot b^-$. Then,   $(a \odot b)^- \odot (a^- \odot b^-) =  ((a \odot b)^- \odot a^-) \odot b^- = (a \odot b)^- \odot b^- = (a \odot b)^-$. Thus $(a \odot b)^- \leq a^- \odot b^-$.

    \item Since $q_j^- = q_j^- \odot q_j \leq q_j^-  \odot  \bigvee_{i \in I} q_i \leq q_j^-$.
    
    \item Since $q_j \leq \bigvee_{i} {q_i}$, from the item $(9)$ we obtain $q_j^- \leq (\bigvee_{i} {q_i})^-$, and then $\bigvee_{j} q_j^- \leq (\bigvee_{i} {q_i})^-$, by sup definition.  
    \end{enumerate}
\end{proof}

\begin{proposition}\label{q-_is_right_adjoint} Consider the maps $i : Idem(\Q) \hookrightarrow \Q$ and $ ( \ )^{-} : \Q \to Idem(\Q)$, then:

\begin{enumerate}
 
\item $(Idem(Q), \bigvee, \odot, 1)$ is a locale and the inclusion map $i:Idem(Q) \hookrightarrow Q$ preserves $\odot$, sups and $\top$.

\item The map $ ( \ )^{-} : \Q \to Idem(\Q)$ preserves $\odot$ and $\top$.

\item The adjunction relations (for posets) holds for each $e \in idem(\Q)$ and $q \in \Q$

$$Hom_{\Q}(i(e), q) \cong Hom_{idem(\Q)}(e, q^-)$$

\end{enumerate}

\end{proposition}

\begin{proof}
\begin{enumerate}
    \item The sup of a set of idempotents is an idempotent (in the same vein of the proof of item $(5)$ in the previous proposition). If $f, e, e' \in Idem(Q)$, then $e \odot e' \leq e, e'$, since $Q$ is semicartesian. Moreover, if $f \leq e, e'$, then $f = f \odot f \leq e \odot e'$. Thus $e \odot e'$ is the g.l.b. of $e, e'$ in $Idem(Q)$. The other claims are straightforward.

    \item This is contained in items $(1)$ and $(10)$ of the previous proposition.

    \item  Since we are dealing with posets, it is enough to show that, for each $q \in Q, e \in Idem(Q)$, 
    $$i(e) \leq q \iff e \leq q^-$$ 
    
    If $i(e) \leq q$, then by item $(6)$ in the previous proposition $e \leq q^-$.
    
    On the other hand, if $e \leq q^-$, then by item $(4)$ and the equivalence in the $(9)$ in the previous proposition $e = e \odot q^-$. Then, by item $(2)$, $e \leq e \odot q \leq q$. 
    
\end{enumerate}

\end{proof}

\begin{definition}
Let $Loc$ be the category of locales morphism that preserves finitary infs and arbitrary sups, and   $CSQ$ the category of commutative semicartesian quantales and morphism that preserves sups and $\top$ satisfying that  $f(a \odot b) \geq f(a) \odot f(b)$.

\end{definition}
 
The next proposition is analogous to the previous one, but for the category $CSQ$ instead of the poset category of commutative and semicartesian categories.

\begin{proposition} Consider inclusion functor $\iota: Loc \hookrightarrow CSQ$. Then:

\begin{enumerate}
 
 \item The inclusion functor $\iota: Loc \hookrightarrow CSQ$ is full and faitful.
 
\item  $\Q \mapsto Idem(\Q)$ determines the right adjoint of the inclusion functor $\iota: Loc \hookrightarrow CSQ$, where   the inclusion $i_Q: \iota (Idem(\Q)) \hookrightarrow \Q$,  is a component of the co-unity of the adjunction.
\end{enumerate}
\end{proposition}

\begin{proof}
\begin{enumerate}
    \item Recall that a locale is a commutative semicartesian quantale where $\odot = \wedge$. It is clear that $\iota$ is a well-defined and faithful functor.   Let $L, L'$ be locales and  $f : \iota(L) \to \iota(L')$ be a $CSQ$-morphism: we must show that it preserves finitary infs. By hypothesis, $f(a \wedge b) \geq' f(a) \wedge' f(b)$. On the other hand, since $f$ preserves sups it is increasing and then  $f(a \wedge b) \leq' f(a) \wedge' f(b)$. Thus $f$ preserves binary infs and $\top$ (by hipothesis). Thus it preserves finitary infs.
    
    \item By item (1) in the previous proposition, $i_Q: \iota (Idem(\Q)) \hookrightarrow \Q$ is a morphism that preserves sups, $\top$ and $\odot$.

Let $\HH$ be a locale and $f : \iota(\HH) \to \Q$   be a $CSQ$-morphism. Let  $a \in \HH$, then $f(a) = f(a \wedge a)  = f(a) \odot f(a)$,  thus $f(a) \in Idem(\Q)$. 
Moreover, $f_| : \HH \to Idem(\Q)$ is a locale morphism, by item  $(1)$. Since $i_Q$ is injective $f_|$ is the unique locale morphism $\HH \to Idem(Q)$ such that $i_Q \circ \iota(f_|) = f$.
\end{enumerate}
\end{proof}
Observe that the last item follows the same idea of  Lemma 2.2 in \cite{nlab:quantale}.

\begin{definition}
A set o elements $\{q_i : i \in I\}$ of $Q$ is a \textbf{partition of $q \in Q$} if $\bigvee_{i \in I} q_i = q$ and $q_i \odot q_j =0$, for each $i \neq j$.
\end{definition}

It is clear that if $\{q_i : i \in I\}$ is a partition of $q$, then $\{q_i \odot a : i \in I\}$ is a partition of $a \odot q$, for any $a \in Q$. Thus, every partition of unity determines a partition for any $q \in Q$.

\begin{example}
For a commutative ring $A$, any ideal $I$ has a partition: Take an idempotent $e \in A$ and observe that $1 = e + 1 - e$. We have $(1) = A$ (the unity), $(e + 1 - e) = (e)+(1-e)$ and $(e)\odot(1-e) = (e)\cap(1-e)=0$. So $\{(e),(1-e)\}$ is a partition of $(1)$ and from it we  obtain a partition $\{(e)\odot I,(1-e)\odot I\}$ for any ideal $I$. If $A$ only have trivial idempotents then the ideals admit trivial partition.
\end{example}

Next we explore other properties of the construction $q \mapsto q^-$ in a more specific class of quantales.

\begin{definition}
We say that a (commutative, semicartesian) quantale $Q$ is: 
\begin{enumerate}
\item  An \textbf{Artinian quantale} if each infinite descending chain $q_0 \geq q_1 \geq q_2 \geq q_3 \geq  \dots$, stabilizes for some natural number $n \in \mathbb{N}$, which may vary according to the chain.
    \item  A \textbf{p-Artinian quantale} if for each $q \in Q$, the infinite descending chain of powers of $q$, $q^1 \geq q^2 \geq q^3 \geq  \dots$, stabilizes for some natural number $n \in \mathbb{N}\setminus \{0\}$, which may vary according to the chain.  
    \item If there is a natural number $n \geq 1$ such that for all $q \in Q$ we have $q^n = q^{n+1}$, then we say that $Q$ is \textbf{uniformly p-Artinian}. The least $n\in \mathbb{N}$ such that, for each $q \in Q$, $q^{n+1} = q^n$ is called the {\bf degree of Q}.
    \end{enumerate}
\end{definition}

The following results are straightforward.

\begin{remark} Let $Q$ be a commutative  and semicartesian quantale. 

\begin{enumerate}
    \item If $Q$ is Artinian or  uniformly p-Artinian, then $Q$ is p-Artinian.
    \item If $Q$ is a p-Artinian quantale, $q \in Q$ and $q^{n+1} = q^n$, then $q^- = q^n$.
\end{enumerate}

\end{remark}

The example that motivates such terminology is  the set of ideals of an Artinian commutative unitary ring.  Concerning this example, we add the following

\begin{proposition}
Let $A$ be a commutative unitary ring and consider $Q=\mathcal{I}(A)$ be its quantale of all ideals. Consider:
\begin{enumerate}
    \item $\mathcal{I}(A)$ is p-Artinian;
    \item For each $a \in A$, there is $n \in \mathbb{N}$ such that $(a)^n = (a)^{n+1}$;
    \item For each $a \in A$, there is $n \in \mathbb{N}$ and $b \in A$ that $a^n = b. a^{n+1}$. This means that $A$ is strongly $\pi$-regular.
    \item Each prime proper ideal of $A$ is maximal;
    \item $A/nil(A)$ is a von Neumann regular ring;

Then we have the following implications
$$1 \implies 2 \iff 3 \implies 4 \iff 5.$$
    Moreover, if $A$ is a reduced ring (i.e. $nil(A) = \{0\}$), then all items above are equivalent between them and also are equivalent to
    \item  $\mathcal{I}(A)$ is a uniformly p-Artinian quantale of degree 2.
\end{enumerate}
\end{proposition}

  \begin{proof}
$1 \implies 2:$ By the definition of p-Artinian.

$2 \iff 3:$ Straightforward.

$3 \implies 4:$ Let $P$
be a prime proper ideal and take $a \notin P$. By $3$, there is $n \in \mathbf{N}$ and $b \in A$ such that $a^n - ba^{n+1} = 0$. So $a^n(1-ba) = 0$. Since $a^n \notin P$ and $P$ is prime, we have that $a^n(1-ba) \in P$. Then $1 \in P + Ra$ and we obtain $A = P + Aa$. In other words, every non-zero element in  $A/P$ is invertible, which means that $A/P$ is field and therefore $P$ is maximal.

$4 \iff 5$ This is stated in \cite[Exercise 4.15]{lam2006exercises}, where Krull dimension $0$ means precisely that  all prime ideals are maximal ideals.

Now, suppose that $\mathcal{I}(A)$ is uniformly $p$-Artinian. This gives that $\mathcal{I}(A)$ is $p$-Artinian and so we do not have verify $1$. We conclude the sequence of implications by showing that $5$ implies $2$: we use that a ring is von Neumann regular iff every principal left ideal is generated by an idempotent element. Since $A$ is commutative and $A/nil(A) = A$ is von Neumann regular, for each $a \in A = A$, $(a) = (e)$ for some idempotent $e \in A$. Therefore, $(a)^n = (e)^n  = (e)^{n+1} = (a)^{n+1} $    

\end{proof}

\begin{example}
Since $A$ is a strongly (von Neumann) regular ring if and only if $A$ is a reduced regular ring \cite[Remark 2.13]{rege1986neumann},
 any reduced regular ring satisfies condition $3$ (every regular ring is $\pi$-regular) so $\mathcal{I}(A)$ is an example of a uniformly p-Artinian quantale of degree 2. 
\end{example}
\begin{remark}
For comutative rings, strongly von Neumann regular is equivalent to von Neumann regular.
\end{remark}

\begin{proposition}
If $Q$ is a uniformly p-Artinian quantale, and $(I, \leq)$ is an upward directed poset, then $(\bigvee_{i \in I} q_i^-) = (\bigvee_{i \in I} q_i)^-$.
\end{proposition}

  \begin{proof}
The relation $(\bigvee_{i \in I} q_i^-) \leq (\bigvee_{i \in I} q_i)^-$ holds in general.

Now suppose that the degree of $Q$ is $n \in \mathbb{N}$. Then 
$$(\bigvee_{i \in I} q_i)^- = (\bigvee_{i \in I} q_i)^n = \bigvee_{i_1, \cdots i_n \in I} q_{i_1} \odot \cdots \odot q_{i_n}.$$
But, since $(\bigvee_{i \in I} q_i)$ is an upward directed sup, for each $i_1, \cdots, i_n \in I,$ there is $j \in I$ such that $q_{i_1}, \cdots, q_{i_n} \leq q_j$ then 
$$\bigvee_{i_1, \cdots, i_n \in I} q_{i_1} \odot \cdots \odot q_{i_n} \leq \bigvee_{j\in I} q_{j}^n = (\bigvee_{j \in I} q_j^-).$$
\end{proof}

Now, observe that the equality $\bigvee_{i\in I} q_i^- = (\bigvee_{i\in I} q_i)^-$ holds, in general (i.e. for each sup) for any locale, but not for any quantale.

\begin{example}\label{ex:non_geometric_quantale}
If $Q = \mathbb{R}_+ \cup \{\infty\}$ is the extended half-line presented in \ref{ex_quantales}, then all elements of $Q$ are in the interval $[0,\infty]$.  There are only two idempotent elements in this quantales, $0$ and $\infty$. 
Since, in this case, the supremum is the infimum, and $0 \geq 0 + q$ if and only if $q = 0$, we have 
\[
q^- = \begin{cases}
0, \mbox{ if } q=0, \\
\infty, \mbox{ if } q \in (0,\infty]
\end{cases}
\]
So, for a subset $\{q_i : q_i \neq 0, \forall i \in I\} \subseteq Q$ then $\bigvee_{i\in I}(q_i^-) = \infty$ but $(\bigvee_{i\in I}q_i)^- $ may be zero or $\infty$ depending if the supremum (which is the infimum in the usual ordering) of $q_i$'s is zero or not.
\end{example}

Since some but not all quantales satisfies such equality, we provide a name for it.
\begin{definition}
Let $Q$ be a commutative semicartesian quantale, we say $Q$ is a \textbf{geometric quantale} whenever 
$\bigvee_{i\in I} q_i^- = (\bigvee_{i\in I} q_i)^-$, for each $\{q_i : i \in I\} \subseteq Q$.
\end{definition}

Locales satisfy, trivially, this geometric condition. Moreover  

\begin{example}
The extended natural numbers presented in \ref{ex_quantales} is a geometric quantale. We argue in \ref{ex:non_geometric_quantale} that the extended positive real numbers is not a geometric quantale because $(\bigvee_{i\in I}q_i)^-$ could be zero. However, we only have  $(\bigvee_{i\in I}q_i)^- = \infty$ since we are considering the subset $\{q_i : q_i \neq 0, \forall i \in I\} \subseteq \mathbb{N}\cup\{\infty\}$.
\end{example}

Note that the poset of all ideals of a PID {\em is not} a geometric quantale. In particular, $(\mathbb{N}, \cdot, \sqsubseteq)$, where $a \sqsubseteq b$ iff $b \mid a$, is not a geometric quantale. 

We choose such terminology to indicate that under those conditions the operation $(-)^-$ is a \textit{strong geometric morphism}, whose definition we provide in Section \ref{sec:comparing_with_Groth_topos}.

Moreover, we may construct geometric quantales from others geometric quantales:

\begin{proposition} The subclass of geometric quantales is closed under arbitrary products and interval construction.
\end{proposition}
  \begin{proof}

Given a family of quantales $Q = \{Q_i:  i \in I\}$ the cartesian product  $\prod_{i \in I}Q_i$ with component-wise order is a geometric quantale.

It follows from the fact that $(-)^-$ is component wise. Indeed,
\begin{align*}
    (q_i)_i^- &=  \bigvee \{(p_i)_i \in Idem(Q) \,:\, (p_i)_i \leq (p_i)_i \odot (q_i)_i\}\\
    &=  \bigvee\{(p_i)_i \in Idem(Q) \,:\, p_i \leq p_i \odot q_i, \forall i \in I\} \\
    &=  (\bigvee\{( p_i \in Idem(Q) \,:\, p_i \leq p_i \odot q_i, \forall i \in I\})_i\\
    &= (q_i^-)_i
\end{align*}

Then
\begin{align*}
   \bigvee_{j \in I}(q_{ij})_i^- =   \bigvee_{j \in I}(q_{ij}^-)_i = (\bigvee_{j \in I}(q_{ij}^-)_i = ((\bigvee_{j \in I}q_{ij})^-)_i
\end{align*}
\end{proof}

We defined $(-)^-$ as a supremum and verified it is the best lower idempotent approximation, in the sense that $q^-$ is the maximum of idempotents $e$ such that $e \leq q$ (Proposition \ref{prop:properties_of_idempotentaprox}.6), or equivalently $e \preceq q$ (Remark \ref{quantale-re}.3) . Analogously, we are tempted to define an operation $(-)^+$ as an infimum and obtain that $q^+$ is the minimum of idempotents $e$ such that $q \leq e$ (or, possibly, $q \preceq e$). For achieve this, we need \textbf{double-distributive quantales}, which are quantales that satisfy the following additional distributive law, for $I \neq \emptyset$: 

$$a \odot (\bigwedge\limits_{i \in I}b_i) = \bigwedge\limits_{i \in I}(a \odot b_i)$$ 
Examples of double-distributive quantales are: locales; the extended half-line $[0,\infty]$ and  the extended natural numbers  $\mathbb{N} \cup \{\infty\}$; a subclass of the quantales of ideals of a commutative and unital ring that is closed under quotients and finite products and contains the principal ideal domains. Besides, double-distributive quantales are closed under arbitrary products and interval construction. 

For members $u \in Q$ these quantales $Q$, there is a more explicit construction of $u^-$ as a transfinite power $u^\alpha$, where $\alpha$ is an ordinal with cardinality $\leq $ cardinality of $Q$.

\begin{proposition} Let $Q$ be a unital double distributive commmutative and semicartesian quantale.
Given $q \in Q$, consider the transfinite chain of powers $(q^\alpha)_{\alpha \geq 1 \ ordinal}$:
$q^1 := q$; $q^{\alpha+1} := q^\alpha \odot q$; if $\gamma \neq 0$ is a limit ordinal, then $q^\gamma := \bigwedge_{\beta< \gamma} q^\beta$
\begin{enumerate}
    \item It is a descending chain;
    \item It stabilizes for some ordinal $\alpha$, $q^\beta = q^{\alpha}$ for each $\beta \geq \alpha$, and $\alpha < \mbox{successor}(card(Q))$;
    \item Moreover, if $q^\alpha = q^{\alpha+1}$, then  $q^- = q^{\alpha}$.
\end{enumerate} 
\end{proposition}
  \begin{proof} 
\begin{enumerate}
    \item This follows directly by induction. 
    \item Suppose that the restriction of the descending chain to all ordinal $\gamma$ with $1 \leq \delta \leq \alpha$ is a strictly descending chain in $Q$.
    Thus we have an injective function $[1, \alpha] \to Q$, $\delta \mapsto q^\delta$. Since $card(\alpha) = card([1, \alpha])$, we must have $card(\alpha) \leq card(Q)$, then $\alpha < \mbox{successor}(card(Q))$. Thus, in particular, there is a largest ordinal $\alpha$ such that $(q^\delta)_{1 \leq \delta \leq \alpha}$ is a strictly descending chain. Thus $q^{\alpha+1} = q^\alpha$ and, by induction, $q^\beta = q^\alpha$ for each $\beta \geq \alpha$.

\item Suppose that the transfinite descending chain stabilizes at $\alpha$ (i.e. $q^{\alpha+1} = q^{\alpha}$). So $q^\alpha = q^\alpha \odot q^\alpha $ and $q^\alpha = q^{\alpha+1}=  q^\alpha \odot q$  and $q^\alpha = q^\alpha \odot q^\alpha $. Thus $q^\alpha$ is an idempotent such that $q^\alpha \preceq q$ (in particular, $q^\alpha \leq q^-$). On the other hand, for any idempotent $p \in Q$ such that $p \preceq q$ (i.e., $p = p \odot q$)  we have, by induction $p \preceq q^\beta$, for all ordinal $\beta \geq 1$: in the induction step for ordinal limits we have to use the hypothesis that $Q$ is double-distributive. 

So $p = p \odot q^\alpha \leq q^\alpha$. Thus $q^\alpha$ is the largest idempotent (in the orders $\leq$ and $\preceq$) such that $q^\alpha \preceq q$. Then, by Proposition \ref{prop:properties_of_idempotentaprox}.(6), $q^- = q^\alpha$.
\end{enumerate}
\end{proof}

Now, we are able to define an upper lower approximation:

\begin{definition}
Let $Q$ be a  commutative and semicartesian quantale that is also unital and "double-distributive". For each for  $q \in Q$, define:
$$q^+  :=  \bigwedge \{p \in Q :q\leq q \odot p\} = \bigwedge \{p \in Q :q\preceq p\}.$$
\end{definition}

\begin{lemma}\label{idempclosure+ lemma} if $Q$ is unital, semicartesian and double-distributive, then

\begin{multicols}{2}
\begin{enumerate}

\item ${1}^+ =1$, ${0}^+ = 0$

\item $q \preceq {q}^+$

\item $q \leq {q}^+$

\item $q \odot {q}^+ = q$

\item $q = {q}^+ \Leftrightarrow q\odot q = q$

\item ${q}^+ \odot {q}^+ = {q}^+$ 

\item ${q^+}^+ = q^+$ 

 \item ${q}^+$ is the $\preceq$-least $y \in Q$  such that $q \preceq y$ and is  $\odot$-idempotent (i.e. $q = q \odot y$ and $y \odot y = y$).

\item $(a \odot b)^+ \leq a^+\odot b^+$ 

\item $x \preceq y$ $\Rightarrow$ $x^+ \preceq y^+$ $\Leftrightarrow$ $x^+ \leq y^+$

Let $\{q_i : i \in I\} \subseteq Q$, then

\item ${q_j}^+ \odot \bigvee_{i \in I} q_i \geq q_j$ 

\item $q_j \odot \bigvee_{i \in I} q_i^+ = q_j$ 

\item $({\bigvee_{j} q_j}) \preceq \bigvee_{i} {q_i}^+$

\item $({\bigvee_{i} q_i})^+ \leq \bigvee_{i} {q_i}^+$ 

Suppose $Q$ divisible, then

\item $x \leq y \Rightarrow x^+ \leq y^+$

\item  $({\bigvee_{i} q_i})^+ = \bigvee_{i} {q_i}^+$

\end{enumerate}
\end{multicols}
\end{lemma}

  \begin{proof}
The proof follows the same spirit of Proposition \ref{prop:properties_of_idempotentaprox}. 
We just check the two properties that require $Q$ to be divisible:

$15.$ Since $Q$ is divisible, $a \leq b$ implies there exist $x \in Q$ such that $a = b \odot x$. Then $a^+ =  (b \odot x)^+ \leq b^+ \odot x^+ \leq b^+$.

$16.$ The non-trivial inequality depends on the hypothesis of $Q$ divisible:
$({\bigvee_{i} q_i})^+ \geq\bigvee_{i} {q_i}^+$. 
This  follows from (15) since $q_j \leq \bigvee_{i} {q_i}$, then $( q_j)^+ \leq (\bigvee_{i} {q_i})^+$ and then, by sup property $\bigvee_j (q_j^+) \leq (\bigvee_{i} {q_i})^+$
\end{proof}

\section{Sheaves on quantales}\label{sec:sheaves_quantales}

We present sheaves on quantales as a functor that satisfies gluing properties, which are formally expressed by an equalizer diagram. This is similar to the definition of sheaves on idempotent quantales proposed in \cite{aguilar2008sheaves}, but it is a completely different case since we are interested in semicartesian quantales: if the quantale was idempotent \textit{and} semicartesian we would obtain a locale, by Proposition \ref{prop:when_quantale_is_locales}, and the theory of sheaves on locales is already well established. 

\textit{From now on we always consider semicartesian quantales, unless stated otherwise.}

Remind that every quantale can be seen as a poset category  $\mathcal{Q}$ where the objects are elements of $\mathcal{Q}$ and the morphism $v \to u$ is given by the order relation $v \leq u$.

\begin{definition}\label{presheaf}
A \textbf{presheaf on a quantale} $\mathcal{Q}$ is a functor $F: \mathcal{Q}^{op} \rightarrow Set$. 
\end{definition}
Given $u,v \in \mathcal{Q}$ such that $v \leq u$, we consider restriction maps $\rho^u_v: F(u) \rightarrow F(v)$ and denote $\rho^u_v(s) = s_{|_v}$, for any $s \in F(u)$.

The functoriality of a presheaf $F$ give us:
\begin{enumerate}
    \item $\forall u \in \mathcal{Q}$ and $\forall s \in F(u), \, s_{|_u} = s$ 
    \item $\forall w \leq v \leq u$ in $\mathcal{Q}$ and $\forall s \in F(u), \, s_{|_w} = (s_{|_v})_{|_w}$
\end{enumerate}
 
We will freely manipulate the restriction maps using the above notation. 

Before we present the definition of sheaves on quantales, we recall the correspondent definition for a locale $\mathcal{L}$.

\begin{definition}\label{def:sheaf_locale}
A presheaf $F: \mathcal{L}^{op}\to Set$ is a \textbf{sheaf} if for any cover $u = \bigvee\limits_{i\in I}u_i$, of any element $u \in \mathcal{L}$, the following diagram is an equalizer

 \begin{center}
     \begin{tikzcd}
F (u) \arrow[r, "e"] & \prod\limits_{i\in I}F (u_i) \arrow[r, "p", shift left=1 ex] 
\arrow[r, "q"', shift right=0.5 ex]  & {\prod\limits_{(i,j) \in I \times I}F (u_i \wedge u_j)}
\end{tikzcd}
 \end{center}
where
\begin{center}
\begin{tabular}{ c c }
 $e(t) = \{t_{|_{u_i}} \enspace | \enspace i \in I\}$, & $p((t_k)_{ k \in I}) = (t_{i_{|_{u_i \wedge u_j}}})_{(i,j)\in I\times I}$ \\ 
  & $q((t_k)_{k \in I}) = (t_{j_{|_{u_i \wedge u_j}}})_{(i,j)\in I\times I}$ 
\end{tabular}
\end{center}
\end{definition}
 
 \begin{remark}
 The cover $u = \bigvee\limits_{i\in I}u_i$ in $\mathcal{L}$ is a cover in the sense of a Grothendieck pretopology. If $\mathcal{L} = \mathcal{O}(X)$ is the locale of open sets of a topological space, it is immediate how to obtain the usual notion of a sheaf on a topological space $X$ from the above definition. We will further explore this remark later.
 \end{remark}

Our approach to define a sheaf on a quantale is simple: we replace the meet operation $\wedge$ by the multiplication $\odot$ of the quantale.

\begin{definition}\label{sheaf_as_functor}
A presheaf $F: \mathcal{Q}^{op}\to Set$ is a \textbf{sheaf} if for any cover $u = \bigvee\limits_{i\in I}u_i$ of any element $u \in \mathcal{Q}$, the following diagram is an equalizer

 \begin{center}
     \begin{tikzcd}
F (u) \arrow[r, "e"] & \prod\limits_{i\in I}F (u_i) \arrow[r, "p", shift left=1 ex] 
\arrow[r, "q"', shift right=0.5 ex]  & {\prod\limits_{(i,j) \in I \times I}F (u_i \odot u_j)}
\end{tikzcd}
 \end{center}
where
\begin{center}
\begin{tabular}{ c c }
 $e(t) = \{t_{|_{u_i}} \enspace | \enspace i \in I\}$, & $p((t_k)_{ k \in I}) = (t_{i_{|_{u_i \odot u_j}}})_{(i,j)\in I\times I}$ \\ 
  & $q((t_k)_{k \in I}) = (t_{j_{|_{u_i \odot u_j}}})_{(i,j)\in I\times I}$ 
\end{tabular}
\end{center}
\end{definition}
 
 \begin{remark}\begin{enumerate}
     \item The cover $u = \bigvee\limits_{i\in I}u_i$ in $\mathcal{Q}$ is not a cover in the sense of a Grothendieck pretopology.
     \item If $\mathcal{Q} = \mathcal{I}(R)$ is the quantale of ideals of a commutative ring with unity $R$ we may mimic the sheaf theory of a topological space but for a ring.
 \end{enumerate}
 \end{remark}

We write, respectively, $PSh(\mathcal{Q})$ and $Sh(\mathcal{Q})$ for the categories of presheaves and sheaves on $\mathcal{Q},$ where the objects are, respectively, presheaves and sheaves, and the morphism are natural transformations between them.
\begin{remark}
The category of sheaves $Sh(\mathcal{Q})$ is a full subcategory of the category of presheaves $PSh(\mathcal{Q})$, that is, the inclusion functor $i :Sh(\mathcal{Q}) \to PSh(\mathcal{Q})$ is full.
\end{remark}
\begin{remark}
The maps $F(u_i) \to F(u_i \odot u_j)$ exist because $u_i \odot u_j $ always is less or equal to $u_i$ and $u_j$, for all $i,j\in I$, since $u_i \odot u_j  \leq u_i \odot 1 = u_i$. This is where we use the semicartesianity.
\end{remark}

Now we develop the firsts steps toward a sheaf theory on quantales, following the presentation of \cite{borceux1994handbook3} in the case of locales.

Let $F$ be a presheaf on $\mathcal{Q}$.
\begin{definition}\label{compatiblefamily}
Let $(u_i)_{i \in I}$ be a family of elements of $\mathcal{Q}.$ We say a family $(s_i \in F(u_i))_{i \in I}$ of elements of $F$ is \textbf{compatible} if for all $i,j\in I$ we have $${s_i}_{|_{u_i \odot u_j}} = {s_j}_{|_{u_i \odot u_j}}$$
\end{definition}

\begin{definition}\label{separatedpresheaf}
We say a presheaf $F$ is \textbf{separated} if, given $u = \bigvee\limits_{i \in I} u_i$ in $\mathcal{Q}$ and $s,s' \in F(u)$, we have $$\big(\forall i \in I\, s_{|_{u_i}} = s'_{|_{u_i}}\big)\implies(s=s')$$
\end{definition}

Using compatible families we equivalently define:
\begin{definition}\label{sheaf_by_comp_fam}
Let  $u = \bigvee\limits_{i \in I} u_i$ in $\mathcal{Q}$ and $(s_i \in F(u_i))_{i \in I}$ a compatible family in $F$, we say the presheaf $F$ is a \textbf{sheaf} if exists a unique element $s \in F(u)$ (called the gluing of the family) such that $s_{|_{u_i}}  = s_i,$ for all $i \in I$.
\end{definition}

It is a straightforward exercise in category theory to show that the definitions \ref{sheaf_as_functor} and \ref{sheaf_by_comp_fam} are equivalent.

\begin{lemma}\label{lema2.1.5}
If $F$ is presheaf on $\mathcal{Q}$, the following conditions are equivalent:
\begin{enumerate}
    \item F is a sheaf.
    \item F is a separated presheaf and given  $u = \bigvee\limits_{i \in I} u_i$ in $\mathcal{Q}$, every compatible family $(s_i \in F(u_i))_{i \in I}$ can be glued into an element $s \in F(u)$ such that $s_{|_{u_i}}  = s_i,$ for all $i \in I$.
\end{enumerate}
\end{lemma}
\begin{proof}
Note that the separated condition $\big(\forall i \in I\, s_{|_{u_i}} = s'_{|_{u_i}}\big)\implies(s=s')$ is the same as the uniqueness condition of the gluing in the sheaf definition.
\end{proof}

\begin{lemma}\label{lema2.1.6}
Let $F$ be a presheaf on a quantale $\mathcal{Q}.$ Then:
\begin{enumerate}
    \item If $F$ is separated, $F(0)$ has at most one element. 
    \item If $F$ is a sheaf, $F(0)$ has exactly one element.
\end{enumerate}
\end{lemma}
\begin{proof}
Consider the empty cover $0 = \bigvee\limits_{i \in \emptyset}u_i$ in $\mathcal{Q}.$ So every family of elements of $F$ is of the form $(s_i \in F(u_i))_{i \in \emptyset}$, so it is immediately compatible.  
Suppose $F$ is separated, then for every $s$ and $s'$ in $F(0)$, the condition $s_{|_{u_i}} = s'_{|_{u_i}}$ holds. Since $F$ is separable, $s=s'$. So, if there is any element in $F(0)$, it is unique.

Suppose $F$ is a sheaf, then there is a unique $s \in F(0)$ such that $s_{|_{u_i}}=s_i, \forall i \in \emptyset.$ Since, in this context, every family is compatible, there is a unique element in $F(0)$, because $s_{|_{u_i}}=s_i$ always holds. 
\end{proof}

\begin{lemma}\label{lema2.1.7}
Let $F$ be presheaf on $\mathcal{Q}$. If $u = \bigvee\limits_{i \in I}u_i$ in $\mathcal{Q}$ and $s \in F(u)$, then the family $\big(s_{|_{u_i}}\big)_{i \in I}$ is compatible.
\end{lemma}
\begin{proof}
Let $i, j \in I$, then $$\big( s_{|_{u_i}}\big)_{|_{u_i \odot u_j}} = s_{|_{u_i \odot u_j}}= \big(s_{|_{u_j}}\big)_{|_{u_i \odot u_j}}  $$
\end{proof}

\begin{lemma}\label{lema2.1.8}
Let $F$ be a sheaf on $\mathcal{Q}$. If, $\forall i,j \in I$ with $i \neq j$, $u_i \odot u_j =0,$ then $F(u) \cong \prod\limits_{i \in I}F(u_i).$
\end{lemma}
\begin{proof}
Since $0 \leq u_i,$ for all $i \in I,$ there is $s \in F(u)$ such that

$$(s_{|_{u_i}})_{|_{u_i\odot u_j}} =  (s_{|_{u_i}})_{|_0} =  s_{|_0} $$
and
$$(s_{|_{u_j}})_{|_{u_i\odot u_j}} =  (s_{|_{u_j}})_{|_0} =  s_{|_0} $$

Besides that, if $i = j$, $(s_{|_{u_i}})_{|_{u_i\odot u_j}} = (s_{|_{u_j}})_{|_{u_i\odot u_j}}$ immediately. So, for each $s \in F(u)$ we have a correspondent compatible family $(s_{u_i} \in F(u_i))_{i \in I}$. Since $F$ is a sheaf, this correspondence is bijective. 

It is clear that every compatible family $(s_i \in F(u_i))_{i \in I}$ is an element of $\prod\limits_{i \in I}F(u_i)$.
On the other hand, every element $(s_i \in F(u_i))_{i \in I}$ of $\prod\limits_{i \in I}F(u_i)$ has a correspondent compatible family, by what we reasoned above. 
\end{proof}

The following construction provides a sheaf over a quantale $\Q$ from a sheaf over $\Q$ and any $u \in Q$. 

\begin{proposition}\label{prop2.3.3}
Let $F$ be a sheaf on a quantale $\cal Q$. For each $w \leq v$, consider:
$$ F_{|_u}(v) = \begin{cases}
F(v), &\mbox{ if } v \leq u \\
\emptyset, &\mbox{ otherwise.}
\end{cases}$$

$$ F_{|_u}(w\to v) = \begin{cases}
F(w\to v): F(v) \to F(w), &\mbox{ if } w \leq v \leq u \\
! : \emptyset \to F_{|_u}(w), &\mbox{ if }  w \leq v \nleq u .
\end{cases}$$

is a sheaf. \footnote{Note that $F_{|_1} = F$ and if $u ' \leq u \in Q$, then ${F_{|_u}}_{|_{u'}} = F_{|_{u'}}$.}
\end{proposition}
\begin{proof}
It is clear that $F_{|_u}$ is a presheaf. Consider $v = \bigvee\limits_{i\in I}v_i$ in $\cal Q,$ and $s, s' \in F_{|_u}(v)$ such that $s_{|_{v_i}}=s_{|_{v_i}}', \, \forall i \in I$

If $v \leq u$, then $s, 's \in F(v) = F_{|_u}(v).$ Since $F$ is a sheaf, it is separated so $s = s'.$ If $v \nleqslant u$, then $s, 's \in \emptyset$ and there is nothing to do. Thus $F_{|_u}$ is a separated presheaf.

Now consider $(s_i \in F_{|_u}(v_i))_{i\in I}$ a compatible family. Suppose $F_{|_u}(v_i) = \emptyset$ for some $i \in I$. For such $i \in I$, there is no $s_i$ in $F_{|_u}(v_i)$, then, there is $j \in I$ such that $s_{i_{|_{u_i \odot u_j}}} \neq s_{j_{|_{u_i \odot u_j}}}$. In other words, the family is not compatible. This implies $F_{|_u}(v_i)=F(v_i),$ for all $i\in I.$ So $v_i \leq u$, which means $\bigvee\limits_{i\in I}v_i = v\leq u$. Therefore $F_{|_u}(v)=F(v)$. 

Since $F$ is a sheaf, we conclude the compatible family $(s_i \in F_{|_u}(v_i))_{i\in I}$ can be glued into $s \in F_{|_u}(v_i)$ such that $s_{|_{v_i}}, \, \forall i \in I$. By \ref{lema2.1.5}, $F_{|_u}$ is a sheaf. 
\end{proof}

Next, we introduce the first (abstract) example of a sheaf on a quantale $\mathcal{Q}$.

\begin{example}

The functor $Q(-, v)$ is a sheaf, for every fixed $v \in \mathcal{Q}$.

Recall that $\mathcal{Q}(-, v)$ is the functor $Hom_{\mathcal{Q}}(-,v)$ so it is a presheaf, where if $w \leq u$, then we send the unique element $\{(u \to v)\}$ in $Q(u,v)$ to the unique element $\{(w \to v) \}$ in $Q(w,v)$. 

Observe that we have two cases:

\begin{enumerate}
        \item Suppose $u \leq v$: since $u_i \leq u$, for all $i \in I$, we have that $u_i \leq v$, for all $i \in I$. Take $s_i = (u_i \to v) \in \mathcal{Q}(u_i,v)$, since $u_i \odot u_j \leq u_i, u_j$, for all $i, j \in I$, $${s_i}_{|_{u_i \odot u_j}} = (u_i \odot u_j \to v)={s_j}_{|_{u_i \odot u_j}}$$
        
        So $(s_i)_{i\in I}$ is a compatible family. To conclude $\mathcal{Q}(u,v)$ is a sheaf, take the only element $s = (u \to  v)  \in \mathcal{Q}(u,v)$ and observe that $ s_{|_{u_i}} = (u_i \to v) = s_i,$ for all $i \in I.$
        
        \item Suppose $ u \nleq v$: if $u_i \leq v$, for all $i \in I$, by definition of supremum, $\bigvee_{i\in I} u_i \leq v,$ which is not possible. So there is at least one $i \in I$ (if $I \neq \emptyset$) such that $u_i \nleq v $. Thus, $\mathcal{Q}(u,v)$ and $\mathcal{Q}(u_i,v)$  are empty sets, for such an $i \in I$. Then the sheaf condition is vacuously true.
        
        If $I = \emptyset,$ then $\bigvee\limits_{i \in \emptyset}u_i = 0$ and $\mathcal{Q}(0,v)$ fits in the first case since $0 \leq v$.
\end{enumerate}
\end{example}

Next we introduce a concrete example of sheaf. 

\begin{proposition}\label{prop.productand interval_of_sheaves}
\begin{enumerate}
    \item 
Let $(Q_j)_{j \in J}$ be a family of commutative and semicartesian quantales and $(F_j)_{j \in J}$ be a family of sheaves, 
 $F_j: Q_j^{op} \to Set $, for each $j \in J$.
 Then: $\prod_{j\in J} Q_j$ is a commutative semicartesian quantal; a family $\{(u_j^i)_{j \in J} : i \in I\}$ is a cover of $(u_j)_{j \in J} \in \prod_{j\in J} Q_j$ iff for each $j \in J$, $\{u_j^i : i \in I\}$ is a cover of $u_j \in Q_j$; and $\prod_{j\in J} F_j: (\prod_{j\in J} Q_j)^{op}\to Set$ given by $(\prod_{j\in J} F_j)(u_j)_{j \in J} := \prod_{j\in J} F_j(u_j)$ is a sheaf with the restriction maps defined component-wise from each $F_j$.
 \item Let $ F : Q^{op} \to Set$ be a sheaf on the commutative and semicartesian quantale $Q$.  Let $e,a \in Q$, $e \leq a$, $e^2 = e$ and consider $Q' = [e,a]$, the (commutative and semicartesian) ``subquantal" of $Q$. Then $F' : {Q'}^{op} \to Set$ defined by $F'(u) = F(u)$, if $u \neq e$ and $F'(e) = \{*\}$, with non-trivial restriction maps $F'(v) \to F'(u) = F(v) \to F(u)$, if $e < v \leq u$,  is a sheaf.
 \end{enumerate}
\end{proposition}
\begin{proof}
\begin{enumerate}
\item  Straightforward.

\item Let $u = \bigvee_{i \in I}u_i$ be a cover in $Q'$, and $(s_i \in F'(u_i))_{i \in I}$ a compatible family. Since $F'(u_i) = F(u_i)$ and the restriction maps for $F'$ are restriction maps for $F$, we have that $(s_i \in F(u_i))_{i \in I}$ a compatible family. Since $F$ is a sheaf, there is a unique gluing $s \in F(u) = F'(u)$. So $F'$ is a sheaf.
\end{enumerate}
\end{proof}

Next we introduce a concrete example of a sheaf. 

\begin{example}\label{ex_ball}
    Take $Q = ([0,\infty],+,\geq)$ the extended half-line quantale. Let $(X,d)$ be an (extended) metric space. For each $A \subseteq X$ and each $r \in [0, \infty]$ consider balls $F_A(r) = B_r(A)=\{x \in X: d(x,A) \leq r\}$. Note that $s \geq r$ entails $B_r(A) \subseteq B_s(A)$ and, in the obvious way $F_A : [0, \infty] \to Set$ became a presheaf over the quantale $Q$ where $F_A((s\geq  r)) : F_A(r) \hookrightarrow F_A(s)$ is the inclusion. Moreover, this is a sheaf, since if $r = \bigwedge_{i \in I} s_i$ in $[0, \infty]$, then the diagram below is an equalizer

$$B_r(A)\rightarrow \prod_i B_{s_i}(A) \rightrightarrows \prod_{i,j} B_{s_i  + s_j}(A)$$ for non-empty coverings. However, if $I = \emptyset$, then $r = \bigwedge_{i \in I} s_i = \infty$. Therefore, $B_{\infty}(A)$ is not single element (i.e., is not the terminal object in $Set$). This means that the sheaf condition fails when $I = \emptyset$. To surpass this,  we maintain our definition $B_r(A)$ for all $r \in [0,\infty)$ but for $r = \infty$ we define $B_\infty(A) = \{*\}$. For any $s \geq r$, the restrictions map is the identity map on $\{*\}$.
\end{example}

The next result is simple but it is fundamental to show that $Sh(Q)$ is a monoidal closed category.

\begin{proposition}\label{expcand-pr}
 Let  $F$ be a sheaf on a quantale $Q$ and $u \in Q$. For each $w \leq v$, the following presheaf\footnote{Note that $F^{(1)} = F$ and if $u',  u \in Q$, then $(F^{(u)})^{(u')} = F^{(u' \odot u)}$.}is a sheaf:
$$ F^{(u)}(v) := F(u \odot v)$$

$$ F^{(u)}(w \leq v) := F(u \odot w \leq u \odot w)$$
\end{proposition}

\begin{proof}
It is clear that $F^{(u)}$ is a presheaf, since $F$ is a sheaf and $w \leq v$ in $Q$ implies that $(u \odot w) \leq (u \odot v)$. 

Note that if $v = \bigvee_i  v_i$ then  $u \odot v  = \bigvee_i  (u \odot v_i)  $ is a cover.

Take a family $(s_i) \in F^{(u)}(v_i) = F(u \odot v_i)$, such that 
$$F(u \odot v_i \odot v_j \leq u \odot v_i )(s_i) = F(u \odot v_i \odot v_j \leq u \odot v_j )(s_j) \in F(u \odot v_i \odot v_j) \, \forall i \in I$$

Since $u \odot v_i \odot u \odot v_j \leq u \odot  v_i \odot v_j $, we have that $s_i \in F(v_i \odot u) $ is a compatible family for $F$. Since $F$ is a sheaf, there is a unique gluing $s \in F(u \odot v) = F^{(u)}(v)$ for the family $(s_i)_{i \in I}$.
\end{proof}

Now we prove categorical properties of $Sh(Q)$ that are classic results in the localic case. We start with a list of statements whose verification follows exact the same steps of the localic case.

\begin{proposition} The subcategory $Sh(\mathcal{Q}) \hookrightarrow Set^{\mathcal{Q}^{op}}$ is closed under limits. 
\end{proposition}
\begin{proof}
Consider a small (index) category $\mathcal{J}$ and a functor $F: \mathcal{J}\to PSh(\mathcal{Q})$ with limits\footnote{All presheaf categories are complete and the limits are computed pointwise} $(L,p_J : L \to F(J))_{J \in \mathcal{J}_0}$. To show that $Sh(\mathcal{Q}) \hookrightarrow Set^{Q^{op}}$ is closed under limits we have to prove that if $F(J)$ is a sheaf for all object $J$ in $\mathcal{J}$, then the limit $L$ is a sheaf. Now, the argument is the verbatim copy of the argument used in the proof of \cite[Proposition 2.2.1]{borceux1994handbook3}, but replacing $J$ and $\mathcal{J}$ by $I$ and $\mathcal{I}$, respectively.  
\end{proof}

\begin{corollary}
 $Sh(\mathcal{Q})$ has a terminal object, the (essentially unique) presheaf such $card(\textbf{1}(u)) =1 $, for each $u \in Q$. Moreover, $Hom_Q(-,1) \cong \textbf{1}$
\end{corollary}
\begin{proof}
Since $Sh(\mathcal{Q}) \hookrightarrow Set^{Q^{op}}$ is closed under limit, a terminal object in $Sh(\mathcal{Q})$ must be the terminal presheaf $Hom_Q(-,1) \cong 1$  ($Hom_Q(u,1) = 1, \forall u \in Q $ ) in $PSh(\mathcal{Q})$. Since $1$ is the top element of $Q$, there is an arrow from $u \to 1$, for all $u \in Q$. So  $card(\textbf{1}(u)) =1 $.
\end{proof}

\begin{proposition}\label{prop:mono_iff_injective}
A monomorphism between sheaves $\eta : F \rightarrowtail G$  is just a monomorphism between their underlying presheaves (and they are monomorphism if and only if $\eta_u : F(u) \to G(u)$ is injective, for each $u \in Q$).
\end{proposition}
\begin{proof}
Again, the argument is the same as in the case of sheaves on locales:  $\eta : F \rightarrowtail G$ is a mono if and only if the pullback of $\alpha$ with itself is
\[\begin{tikzcd}
	G & G \\
	G & F
	\arrow["\alpha", from=1-2, to=2-2]
	\arrow["\alpha"', from=2-1, to=2-2]
	\arrow["{id_G}", from=1-1, to=1-2]
	\arrow["{id_G}"', from=1-1, to=2-1]
	\arrow["\lrcorner"{anchor=center, pos=0.125}, draw=none, from=1-1, to=2-2]
\end{tikzcd}\]
In other words, $\eta : F \rightarrowtail G$ is a mono if and only $(G,id_G,id_G)$ is the kernel pair of $\alpha$. This holds in any category. So $\eta$ is a mono of sheaves iff the kernel pair of $\alpha$ is $(G,id_G,id_G)$ iff $\eta$ is a mono of the underlying presheaves. 
Next, we use another general fact about categories. For a small category $\mathcal{A}$ and a category with pullbacks $\mathcal{B}$, a morphism $\beta : F \to G $ in the functor category $Func(\mathcal{A},\mathcal{B})$ is mono iff all the components $\beta_a  : F(a) \to G(a)$ is injective for all object $a$ of $\mathcal{A}$ \cite[Corollary 2.15.3]{borceux1994handbook}. Take $\mathcal{A} = Q^{op}$ and $\mathcal{B}  = Set$, then we obtain that $\eta$ is mono iff $\eta_u$ is injective for all $u \in Q.$
\end{proof}

\begin{corollary}
  Every morphism $\eta : \Q(-,v) \to F$, where $F$ is a (pre)sheaf is, automatically, a monomorphism.
\end{corollary}
\begin{proof}
By Proposition \ref{prop:mono_iff_injective}, is enough to show that $\eta_u : Q(u,v) \to F(u)$ is injective for all $u\in Q$. This always holds because $Q(u,v)$ has at most one element.
\end{proof}

\begin{proposition}
$Sh(\mathcal{Q})$ has a set of generators.
\end{proposition}\label{proposition:Sh(Q)_has_generators}
\begin{proof}
The set of generators for $Sh(\mathcal{Q})$ is the family of representable sheaves $\mathcal{Q}(-,u)$, indexed by elements  of $\mathcal{Q}$.

Take $\eta, \eta' : H \to F$ two distinct morphisms of sheaves. Consider the index $u$ as the infimum element of $\mathcal{Q}$. Observe the following composition
\[\begin{tikzcd}
	{\mathcal{Q}(-,u)} & H & F
	\arrow["\zeta", from=1-1, to=1-2]
	\arrow["{\eta'}"', shift right=1, from=1-2, to=1-3]
	\arrow["\eta", shift left=1, from=1-2, to=1-3]
\end{tikzcd}\]
Since there is no element smaller than $u$, the only component in which makes sense calculate $\zeta$ is the element $u$ himself. So we only have
\[\begin{tikzcd}
	{\mathcal{Q}(u,u)} & H & F
	\arrow["{\zeta_u}", from=1-1, to=1-2]
	\arrow["{\eta'_u}"', shift right=1, from=1-2, to=1-3]
	\arrow["{\eta_u}", shift left=1, from=1-2, to=1-3]
\end{tikzcd}\]
where $\mathcal{Q}(-,u)$ is (naturally isomorphic to) the identity map $u \to u$. Since $\eta \neq \eta '$, we conclude $\eta \circ \zeta \neq \eta' \circ \zeta$, as desired.

\end{proof}

\begin{proposition} \label{hom-proposition}
\begin{enumerate}
    \item For each $v,v' \in \Q$, there is at most one (mono)morphism $\Q(-,v) \to \Q(-, v')$ and this exists precisely when $v \leq v'$. 
    \item If $H$ is a sheaf and $\epsilon : H \rightarrowtail Q(-,v)$ is a monomorphism, then  $H \cong Q(-, h)$ where $h = \bigvee\{u \leq v: H(u) \neq \emptyset\}$.
\end{enumerate}
\end{proposition}

\begin{proof} 
\begin{enumerate}
    \item For each $u,v\in \Q$, note that $card (\Q(u,v)) \in \{0,1\}$. \\
Suppose there is a morphism $\eta : \Q(-,v) \to \Q(-,v')$. So, for all $u \in Q$ we have $\eta_u : \Q(u,v) \to \Q(u,v')$. If $\Q(u,v') = \emptyset$, then $\Q(u,v) = \emptyset$. Thus if $u \leq v$, then $u \leq v'$. In particular, for $u = v$, we obtain $v \leq v'$.

Conversely, if $v \leq v'$, consider $i_{v,v'} : \Q(-,v) \to \Q(-,v')$.  For all $u \in Q$ we have $i_{v,v'}(u) : \Q(u,v) \to \Q(u,v')$. 

If $u \nleq v$, then $Q(u,v) = \emptyset$ and $i_{v,v'}(u) : \emptyset \to Q(u,v')$ the unique function from the $\emptyset$, since the $\emptyset$ is an initial object in $Set$.

If $u \leq v$, since $v\leq v'$, $u \leq v'$ and then $i_{v,v'}(u) (u \to v) = (u \to v')$. For any other morphism $j_{v,v'} : \Q(-,v) \to \Q(-,v')$ we obtain that $j_{v,v'}(u) : \emptyset \to Q(u,v')$ the unique function from the $\emptyset,$ whenever $u \nleq v$ and $j_{v,v'}(u) (u \to v) = (u \to v')$, whenever $u\leq v$. So $i_{v,v'} = j_{v,v'}$.

\item Since $\epsilon$ is a monomophism, $\epsilon_u$ is injective and then $card(H(u)) \in \{0,1\}$ for each $u \in Q$ with $H(u)  = \emptyset$ whenever $u \nleq v$. So let $$h = \bigvee\{u \leq v: H(u) \neq \emptyset\} =  \bigvee\{u \in Q: H(u) \neq \emptyset\}.$$ 

We will show that $H(u)$ is non-empty only when $u \leq v$. Note that:

- If $q \leq p$ and $H(p) \neq \emptyset$, then $H(q) \neq \emptyset$ (since $H$ is a presheaf);

- Since $card(H(h)) =1$, we have $H(h)\neq \emptyset$. Once   
$H(p), H(q) \neq \emptyset$ entails $H(p \odot q) \neq \emptyset$, by the sheaf condition we have an equalizer diagram between two parallel arrows where the source and target are both singletons.

Therefore, $H(u) \neq \emptyset$ iff $u \leq h$. Now, we will show that $H(u) \to Q(u,h)$ is a (unique) bijection, for each $u \in Q$.

If $u \nleq h$, then $\emptyset = H(u) \to Q(u,h) =\emptyset$. If $u \leq h$, then $H(u)$ and $Q(u,h)$ are both singletons. So $\epsilon_u$ is an injection and a surjection in $Set$, therefore, a bijection for all $u \in Q$ and then $\epsilon$ is an isomorphism. 
\end{enumerate}
\end{proof}

\section{Sh(Q) is not a topos}\label{sec:comparing_with_Groth_topos}

We remind the reader that a Grothendieck topos is any category equivalent to the category $Sh(\mathcal{C},J)$ of sheaves on a category $\mathcal{C}$ with a Grothendieck (pre)topology $J$. The pair $(\mathcal{C},J)$ is called a site. In particular, $Sh(L)$ is a Grothendieck topos where $\mathcal{C} = L$ is the poset category by the locale $L$ and we define the  Grothendieck  pretopology by $\{f_i: U_i \to U\} \in J(U) \iff u = \bigvee_{i\in I}u_i$. However, if  $u = \bigvee_{i\in I}u_i$ and $v \leq U$ for $u, u_i, v$ in a semicartesian quantale, we may have $ v \neq  u\odot v  =\bigvee_{i\in I} u_i \odot v$. This means that the \textit{stability axiom} in the definition of a Grothendieck pretolopogy is not satisfied in the quantalic case. Nevertheless, there could exist another to provide that $Sh(Q)$ is a Grothendieck topos.

In this section we study deeper categorial properties that makes $Sh(Q)$ even more similar to a Grothendieck topos but we also show that it is not a Grothendieck topos. The argument relies on the fact every Grothendieck topos is an \textit{elementary topos} --- cartesian closed category with pullbacks, a terminal object, and a subobject classifier. It is well known that if $A$ is an object of a topos $\mathcal{E}$, then $Sub_{\mathcal{E}}(A)$ is a Heyting Algebra\footnote{The class of complete Heyting algebras and of locales coincide.} \cite[Proposition 6.2.1]{borceux1994handbook3}. We will prove that $Sub_{Sh(Q)}(Q(-,a)) \cong [0,a]$ is an isomorphism of quantales. Since $[0,a]$ is not a Heyting Algebra in general when $a$ is not an idempotent element, then $Sh(Q)$ is not a topos. 

The first property we want for $Sh(Q)$ is that it has a set $(G_i)_{i \in I}$ of strong generators. Remind that: 
\begin{definition}\cite[Definition 5.2.1]{borceux1994handbook2}
A category $\mathcal{M}$ is locally $\lambda$-presentable, for a regular cardinal $\lambda$, when
\begin{enumerate}
    \item $\mathcal{M}$ is cocomplete;
    \item $\mathcal{M}$ has a set $(G_i)_{i \in I}$ of strong generators;
    \item each generator $G_i$ is $\lambda$-presentable. 
\end{enumerate}
\end{definition}

It is known that any Grothendieck topos is a $\lambda$-locally presentable category for some regular cardinal $\lambda$ \cite[Proposition 3.4.16]{borceux1994handbook3}. We were able to show that $Sh(Q)$ also is $\lambda$-locally presentable, where $\lambda = max\{card(Q)^+, \aleph_0\}$, but we will let a detailed proof about this for a future paper. We do sketch the argument for the reader used with the terminology: use $y$ to denote the Yoneda embedding and consider a covering $\{u_i : i \in I\}$. The Day convolution gives a monoidal structure in $PSh(Q)$ such that we have ``projections'' $y(u_j) \otimes y(u_k)  \xrightarrow{p_1^{jk}} y(u_j)$ and $y(u_j) \otimes y(u_k) \xrightarrow{p_2^{jk}} y(u_k)$. Since $u_i \leq u$, for each $i\in I$ ,there is an arrow $y(u_i) \xrightarrow{\phi_i} y(u)$. Denote by $y(u_j)\underset{y(u)}\otimes y(u_k)$ the equalizer of $\phi_j \circ p_1^{jk}$ and $\phi_k \circ p_2^{jk}$. Thus we define a sieve $S(\{u_i\})$ of a covering $\{u_i : i \in I\}$ as the following coequalizer in $PSh(Q)$

$$\xymatrix{\coprod\limits_{j,k}y(u_j)\underset{y(u)}\otimes y(u_k)\ar@<.7ex>[r]^-{}\ar@<-.7ex>[r]_-{} & \coprod\limits_{i}y(u_i)\ar[r]^{} & S(\{u_i\}) }$$

with the coproduct on the left being taken over $y(u_j) \underset{y(u)}\otimes y(u_k) \rightarrowtail y(u_j)\otimes y(u_k) \cong y(u_i \odot u_j)$.
Then we orthogonalize the class of morphism $\{S(\{u_i\}) \to y(u) : u = \bigvee\limits_{i\in I}u_i\}$ and observe that such orthogonalization corresponds to the category of sheaves on $Q$. This shows that $Sh(Q)$ is a $\lambda$-orthogonality class in $PSh(Q)$. By the Theorem $1.39$ of \cite{adamek1994locally}, we conclude 

\begin{corollary}
    $Sh(Q)$ is:
    \begin{enumerate}
        \item a reflective subcategory of $PSh(Q)$ closed under $\lambda$-directed colimits,
        \item locally $\lambda$-presentable
    \end{enumerate}
\end{corollary}
The first item provides that the inclusion functor $i: Sh(Q) \to PSh(Q)$ has a left adjoint functor $a: PSh(Q) \to Sh(Q)$, which we will call \textbf{sheafification functor}. The second item provides that $Sh(Q)$ has a set of strong generators. 

First, we use the sheafification to show that $Sh(Q)$ admits a monoidal (closed) structure. We will use the notions of \textit{normal reflective embedding} and \textit{normal enrichment for a reflective embedding}. The definitions are available at \cite{day1973note} but we reproduce they here with a different notation: 
\begin{definition}   
    Let $\psi \dashv \phi: D \to B$ be an adjoint pair.
    \begin{enumerate}
        \item $\psi \dashv \phi$ is a  reflective embedding if $\phi$ is full and faithful on morphism.
        \item When $B$ has a fixed monoidal closed structure the reflective embedding is called \textbf{normal} if there exists a monoidal closed structure on $D$ and monoidal functor structures on $\psi$ and $\phi$ for which $\phi$ is a normal closed functor and the unit and counit of the adjunction are monoidal natural transformations.
    \end{enumerate}
\end{definition}

Given that $B$ is a category with a fixed monoidal structure, by saying that a reflective embedding $\psi \dashv \phi: D \to B$ \textit{admits a  normal enrichment} we mean that there are conditions for $\psi \dashv \phi$ be normal. In particular, the functor $\psi$ carries the (closed) monoidal structure from $B$ to $D$ in a compatible and unique (up to monoidal isomorphism) way, see \cite{day1973note} for details.  Also, the reader may find the definition of a normal closed functor in \cite{barr1969adjunction} or be satisfied by the statement that the inclusion functor is a normal closed functor. Then we state  Proposition  1.1 in \cite{day1973note}, with a different notation.

\begin{proposition}\label{prop:3.1_day_monoidal_localization}
Let $C = (C,\otimes)$ be a small monoidal category.  A reflective embedding $\psi \dashv \phi: D \to PSh(C)$ admits
normal enrichment if and only if the functor $F(U\otimes -)$ is isomorphic to some object in $D$ whenever $F$ is a object of $D$ and $U$ is an object of $C$.
\end{proposition}

\begin{proposition}
    The sheafification $a: PSh(Q) \to Sh(Q)$ admits a normal enrichment.
\end{proposition}
\begin{proof}
    In this case, $C = Q$ is the posetal category of quantales. By Propositon \ref{expcand-pr}, the functor $F(u \odot -)$ is a sheaf for every $u \in Q$, whenever $F$ is a sheaf.
    By \ref{prop:3.1_day_monoidal_localization}, the reflective embedding $a \dashv i$ admits a normal enrichment.
\end{proof}

\textbf{Monoidal structure in $ Sh(Q)$:} The above result gives that $Sh(Q)$ has a monoidal closed structure where  $F \otimes G := a(i(F) \otimes_{Day} i(G))$, for $F,G$ sheaves on $Q$. 

Now, we state two results from Borceux to obtain more information about $Sh(Q)$:

\begin{proposition}\cite[Proposition 4.5.15]{borceux1994handbook}
If $\mathcal{C}$ has finite limits and possesses a strong set of generators, so $\mathcal{C}$ is well-powered (i,e., for all $C$ object of $\mathcal{C}$, the subobjects  $Sub(C)$ of $\mathcal{C}$ forms a set).
\end{proposition}
\begin{proposition}\cite[Corollary  4.2.5]{borceux1994handbook}
In a complete and well-powered category, $Sub(C)$ has all infima/intersections and suprema/unions.
\end{proposition}

\begin{corollary}
  $Sh(Q)$ is a complete and well-powered category, and for all $F$ sheaf on $Q$, $Sub(F)$ has all infima/intersections and suprema/unions.
\end{corollary}

\begin{corollary}

{\bf Factorization of morphisms in} $Sh(Q)$:

For each morphism $\phi: F \to G$ in $Sh(Q)$, there exists the least subobject of $G$, represented by $\iota : G' \rightarrowtail G$, such that $\phi = \iota \circ \phi'$ for some (and thus, unique) morphism $\phi' : F \to G'$. Moreover, $\phi'$ is an {\em epimorphism}.

\end{corollary}
  \begin{proof}
By the previous results, there exists the  extremal factorization above $\phi = \iota \circ \phi'$, such that $\iota : G' \rightarrowtail G$ is a mono. To show that $\phi' : F \to G'$ is an epi, consider $\eta, \epsilon: G' \to H$ such that $\eta \circ \phi' =  \epsilon \circ \phi'$ and let $\gamma  : H' \rightarrowtail G'$ be the equalizer of $\eta, \epsilon$. Then, by the universal property of $\gamma$, there exists a unique $\phi'' : F \to H'$ such that $\gamma \circ \phi'' = \phi'$. On the other hand, by the extremality of $\iota$, there exists a unique $\gamma ' : G' \to H'$ such that $\iota = \iota \circ \gamma \circ \gamma'$. Since $\iota$ is a mono, we obtain that $\gamma \circ \gamma' = id_{G'}$, Thus $\gamma$ is a mono that is a retraction: this means that $\gamma = eq(\eta, \epsilon)$ is an iso, i.e., $\epsilon = \eta$. Thus $\phi'$ is an epi.
\end{proof}

\begin{remark} \label{balanced-re} Keeping the notation above, if $\phi : F \to G$ is already a mono then, by the extremality of $\iota : G' \rightarrowtail G$, $\phi \cong \iota$ and thus $\phi' : F \to G'$ is an isomorphism.  It is natural to ask ourselves if the converse holds in general. Conversely, does it hold that any morphism that is mono and epi is an iso? This would mean that the category $Sh(Q)$ is {\em balanced}. 

Any category with factorizations (extremal mono, epi) and  where all the monos are regular (i.e., monos are equalizers) is balanced. A ``topos-theoretic" way to show that all monos are regular is to show that there exists a ``universal mono" $true : 1 \rightarrowtail \Omega$ that is a  subobjects classifier. We will address this question in the next section.

\end{remark}

\begin{theorem} \label{isoposet-th} Assume that $Q$ is unital.
We have the following isomorphisms of complete lattices: 
\begin{align*}
  h_Q\colon  Q & \to  Represented(Sh(Q))  \\
  q & \mapsto Q(-, q) 
\end{align*}
\begin{align*}
    i_Q\colon   Represented(Sh(Q)) &\to Representable(Sh(Q))/isos \\
Q(-,q) & \mapsto [Q(-, q)]_{iso}
\end{align*}
\begin{align*}
    j_Q\colon Representable(Sh(Q))/isos & \to Sub(\textbf{1}) \\ 
    [R]_{iso} & \mapsto [R \cong Q(-,q) \rightarrowtail Q(-,1) \cong \textbf{1}]_{iso}
\end{align*}

Thus
$k_Q = j_Q \circ i_Q \circ h_Q: Q \to Sub(\textbf{1})$ is an isomorphism of complete lattices.  

More generally, take any $a \in Q$, we may amend the  map $k_Q$ in a way that it sends $b \in [0,a]$ to $ [Q(-,b) \rightarrowtail Q(-,a)]_{iso}$, then we obtain a quantalic isomorphism $k_a : [0,a] \to Sub(Q(-,a))$.

\end{theorem}

  \begin{proof}
We will just show that $h_Q, i_Q, j_Q$ are isomorphisms of posets, and, since $Q$ is a complete lattice, then $h_Q, i_Q, j_Q$ are complete lattices isomorphisms.

\textit{$h_Q$ is isomorphism:}  By the very definition of represented functor, the map $h_Q$ is surjective. For injectivity see that  $Q(-,q) = Q(-,p)$ implies that  $Q(u,q) = Q(u,p)$, for all $u \in Q$, and so $p=q$. Yoneda's lemma and Proposition  \ref{hom-proposition}.1 establishes that it preserves and reflects order since $p\leq q$ iff there is some (unique) (mono)morphism $\eta : Q(-,p) \to Q(-,q)$.

\textit{$i_Q$ is isomorphism:} Since it is a quotient map, it is surjective. $i_Q$ is injective: by Proposition \ref{hom-proposition}.2, $Q(-,p) \cong Q(-,q)$ implies $p = q$ and thus $Q(-,p) = Q(-,q)$. The map preserves and reflects order: this is a direct consequence of  Proposition \ref{hom-proposition}.1. 

\textit{$j_Q$ is isomorphism:} Since $! : Q(-,1) \to \textbf{1} $ is an isomorphism, we will just prove that \  
$j'_Q\colon Representable(Sh(Q))/isos \to Sub(Q(-,1))$ \
    $[R]_{iso}  \mapsto 
    [R \cong Q(-,q) \rightarrowtail Q(-,1)]_{iso}$ is an isomorphism. By the very definition of $Sub(F) = Mono(F)/isos$,  it is clearly injective. Take take $\eta : R \rightarrowtail Q(-,1)$, by Proposition \ref{hom-proposition}.1, $R \cong Q(-,q)$, thus $j'_Q$ is surjective. Now let $R$ and $R'$ be representable functors, there is a morphism $\eta : R \to R'$ iff this morphism is unique and it is a monomorphism, thus $j'_Q$ preserves and reflects order.

\end{proof}

\begin{definition}
For each $F$ sheaf on $Q$, we define the following binary operation on $Sub(F)$: Given $\phi_i : F_i \rightarrowtail F$, $i =0,1$ define $\phi_0 * \phi_1 : F_0 * F_1  \rightarrowtail F$ as the mono in the extremal factorization of $F_0 \otimes_F F_1 \rightarrowtail F_0 \otimes F_1 \rightrightarrows F$.
\end{definition}

\begin{theorem}\label{subobj-th}
For each $a \in Q$, the poset $Sub(Q(-,a))$, endowed with the binary operation $*$ defined above  is a commutative and semicartesian quantale. Moreover the poset isomorphism $k_a: [0,a] \to Sub(Q(-,a))$, $q \mapsto [Q(-, q)]_{iso}$, established in Theorem \ref{isoposet-th}, is a quantale isomorphism.
\end{theorem}

  \begin{proof} 

     As a consequence of the proof of Theorem \ref{isoposet-th}, this map is well-defined, bijective and preserves and reflects orders, thus it is a complete lattice isomorphism. It remains to show that $Q(-,u \odot v) \cong Q(-, u) * Q(-,v)$, for all $u, v \leq a$. We have that  $Q(-, u) * Q(-,v) \rightarrowtail Q(-,a)$ is the mono in the extremal factorization of the arrow
$$Q(-,u) \otimes_{Q(-,a)} Q (-,v)  \xrightarrow{equ} Q(-,u) \otimes Q (-,v)  \rightrightarrows Q(-,a)$$

By Day convolution,
$$Q(-,u) \otimes Q (-,v) \cong Q(-,u\odot v)$$

Since $u \odot v \leq a$, by Proposition \ref{hom-proposition}, there is unique (mono)morphism $Q(-,u \odot v)  \rightarrow Q(-,a)$. 
So $Q(-,u) \otimes Q (-,v)  \rightrightarrows Q(-,a)$ corresponds to $Q(-,u \odot v)  \rightrightarrows Q(-,a)$. Thus, the parallel arrows coincide and then
$$Q(-,u) \otimes_{Q(-,a)} Q (-,v)  \cong Q(-,u) \otimes Q (-,v)  \cong Q(-,u \odot v)$$

Hence, the arrow
$$Q(-,u) \otimes_{Q(-,a)} Q (-,v)  \to Q(-,a) $$

is isomorphic with the unique mono
$$Q(-,u \odot v)  \rightarrowtail Q(-,a). $$

This shows that $Q(-,u \odot v) \cong Q(-, u) * Q(-,v)$, as we wish.

\end{proof}

This has an interesting direct application:

\begin{corollary}
    Let $Q$ be the quantale of ideals of a ring $R$, then $Q$ is isomorphic to $Sub(Q(-,R))$    
\end{corollary}

So we can recover any ideal of $R$ by analyzing the subobjects of the sheaf $Q(-,R)$.

\vspace{0.5cm}

{\bf Summarizing this section:}\\ 
Take any commutative, semicartesian and unital quantale $Q$ that has no underlying localic structure --  for example,  let $\mathds{Z}[x]$ be the ring of integer polynomials. This is an example of a ring that is not a \textit{Prüfer domain}, thus we do not have $I\cap(J+K)=I\cap J+I\cap K$, for nonzero ideals. In other words, the quantale of ideals $\mathcal{I}(\mathds{Z}[x])$ is an example of a quantale with no underlying localic structure -- and consider the sheaf $\textbf{1} \cong Q(-,1)$. By the Theorems \ref{isoposet-th} and \ref{subobj-th}, $Sub(Q(-,1)) \cong Sub(\textbf{1})$ is isomorphic to $Q$ as a quantale. Since $Q$ is not a locale, then $Sub(\textbf{1})$ is not a Heyting algebra (locale = complete Heyting algebra). Therefore, from a well-known result in topos theory, we can conclude that $Sh(Q)$ is not even an elementary topos.

\section{On the subobject classifier} \label{sec:truthvalues}

The subobject classifier in a category may be seen as it \textit{internal truth values object}. By definition, every elementary topos has a subobject classifier and we use it to construct the internal logic of a topos. In the category on sheaves on a locale $L$, the subobject classifier is the sheaf $\Omega(u) = \{q \in L : q \leq u\}$ such that for all $v \leq u$, we map $q$ to $q\wedge v$, and $\top : \textbf{1} \to \Omega$ defined by $\top_u(*) := u$ is the ``universal subobject". Thus, for every $F$  sheaf on $L$ we have a natural isomorphism $\eta_F : Sub(F) \to Sh(L)(F, \Omega)$ that sends equivalence class of monos $m: S \to F$ to its unique characteristic map $\chi_m : F \to \Omega$ (the proof of \cite[Proposition 1, Chapter I.3]{maclane1992sheaves} verifies that $\eta_F$ is a natural iso).  In particular $Sh(L)(\textbf{1}, \Omega) \cong Sub(\textbf{1}) \cong L$  explains how $\Omega$, which is the ``internal truth values object" of $Sh(L)$,  encodes the ``external  truth values" of $Sh(L)$, the locale $L \cong Sub(\textbf{1})$.

For general semicartesian and commutative quantales, just replacing the infimum by the quantalic multiplication does not yields a sheaf. Here we present two constructions -- $\Omega^-$ and $\Omega^+$ -- that provides approaches of a subobject classifier in $Sh(Q)$, in different subclasses of commutative semicartesian quantales, and pointing advantages and drawbacks of each one.

\begin{proposition}\label{classifiercand(-)-ex}
Let $Q$ be a commutative, semicartesian and  {\em geometric} quantale. For each $u \in Q$ define $\Omega^-(u) = \{q \in Q\,:\, q \odot u^- = q\} $
then, with the restriction map 
\begin{align*}
    \Omega^-(u) &\rightarrow \Omega^-(v)\\
     q\,\,\,\, &\mapsto q \odot v^-
\end{align*}
 for all $v\leq u$ in $Q$, $\Omega$ is a sheaf.
\end{proposition} 
  \begin{proof}
Note that $q \odot v^- \in \Omega^-(v)$ since $q \odot v^- \odot v^- = q \odot v^-$. It is a presheaf because $q \odot u^-  = q$ and, given $w \leq v \leq u$, $q \odot v^- \odot w^- = q \odot w^-$. The separability also holds: suppose $u = \bigvee_{i \in I} u_i $ and take $p, q \in \Omega^-(u)$ such that $p_{\restriction_{u_i}} = q_{\restriction_{u_i}} $ for all $i \in I.$ Then
\begin{align*}
    p &= p\odot u^- = p\odot (\bigvee_{i \in I} u_i)^- = p \odot \bigvee_{i \in I} u_i^- = \bigvee_{i \in I}p \odot  u_i^-\\
    &  = \bigvee_{i \in I}q \odot  u_i^- = q \odot \bigvee_{i \in I} u_i^- = q \odot (\bigvee_{i \in I} u_i)^- =  q\odot u^- \\
    & = q
\end{align*}

The gluing is $q = \bigvee_{i \in I} q_i$, where $q_i \in \Omega^-(u_i)$ for each $i \in I$. Observe that $q \in \Omega^-(u)$:

$$q \odot u^- = \bigvee_{i \in I}q_i \odot (\bigvee_{j \in I}u_j)^- = \bigvee_{i \in I}q_i \odot \bigvee_{j \in I}u_j^- =\bigvee_{i \in I}q_i \odot u_i^- \odot \bigvee_{j \in I}u_j^- = \bigvee_{i \in I} q_i = q. $$

Where we used that the quantale is geometric in the second equality, the fact $q_i \in \Omega^-(u_i)$ in the third, and the idempotence of $u_i^-$ in the forth.

Now we check that $q$ is the gluing. 
One the one hand 
$$q_j = q_j \odot u_j^- \leq q \odot u_j^- = q_{\restriction_{u_j}} $$

On the other hand, recording that $(u \odot v)^- = (u^- \odot v^-)$ by Proposition \ref{prop:properties_of_idempotentaprox}.10,
\begin{align*}
     q_{\restriction_{u_j}} &= q \odot u_j^- = (\bigvee_{i \in I}q_i) \odot u_j^- = \bigvee_{i \in I}(q_i \odot u_j^-) = \bigvee_{i \in I}(q_i \odot u_i^- \odot u_j^-) \\
     &= \bigvee_{i \in I}(q_i \odot (u_i \odot u_j)^-) = \bigvee_{i \in I}q_{i_{\restriction_{u_i\odot u_j}}} = \bigvee_{i \in I}q_{j_{\restriction_{u_i\odot u_j}}} \\
     &= \bigvee_{i \in I}(q_j \odot (u_i \odot u_j)^-) = \bigvee_{i \in I}(q_j \odot u_i^- \odot u_j^-) = (\bigvee_{i \in I}u_i^-) \odot q_j \odot u_j^-\\
     &=(\bigvee_{i \in I}u_i)^-\odot q_j = u^- \odot q_j \leq q_j.
\end{align*}
\end{proof}
\begin{remark}
\begin{enumerate}

\item The  mapping $Q \mapsto \Omega^-$  preserves products and interval constructions (see Proposition \ref{prop.productand interval_of_sheaves}).

\item Note that for each $v, u \in Q$, such that $v^- = u^-$, then $\Omega^-(v) = \Omega^-(u)$. In particular, if $u^-\leq v\leq u$, then $\Omega^-(v) = \Omega^-(u)$ and, moreover, $\Omega^-(u^-,v) = \Omega^-(v, u) = id_{\Omega^-(v)}$.

 \item  For each $u \in Q$, let $\bot_u, \top_u : 1(u) \to \Omega^-(u)$, where  where $\bot_u(*) := 0 \in \Omega^-(u)$ and $\top_u(*) :=  u^- \in \Omega^-(u)$. Then $\bot, \top : 1 \to \Omega^-$ are natural transformations. 

    \item For each $u \in Q$ and $v \in \Omega^-(u)$ we have $v^- \in \Omega^-(u)$: this defines a map $-_u : \Omega^-(u) \to \Omega^-(u)$, Then $( \ )^- :=(-_u)_{u \in Q} : \Omega^- \to \Omega^-$ is a natural transformation and $\top^- := (\ )^- \circ \top = \top$, $\bot^- := (\ )^- \circ \bot = \bot$.

    \item 
    If $Q$ is a locale, then $\Omega^-(u) = \{q \in Q \, : \, q \odot u^- = q\} = \{q \in Q \,:\, q \leq u\} = \Omega_{0}(u)$, and $\top_u(*) = u^- = u$. Thus  $\top : 1 \to \Omega^-$ coincides with the subobject classifier in the category of sheaves on locales \cite[Theorem 2.3.2]{borceux1994handbook3}. We will readdress  this subject bellow.
\end{enumerate}
\end{remark}

Our investigations did not lead to $\Omega^-$ being a  subobjects classifier, but it does classifies the dense subobjects:  

\begin{definition}
A morphism of sheaves $\eta : G \to F$ is {\bf dense} whenever
$\forall u \in Q \forall y \in F(u) \exists m \in Q, u^-\leq m \leq u$ such that $F(m \leq u)(y) \in range(\eta_m) $ iff $y \in range(\eta_u)$
\end{definition}

Note that, since $m \leq u$, $y \in range(\eta_u)\implies F(m\leq u)(y) \in range(\eta_m)$.

It can be easily verified that a sufficient condition to a morphism of sheaves $\eta : G \overset\cong\to F$ be a dense is:
$\forall u \in Q  \exists m \in Q, u^-\leq m \leq u$ such that the diagram below is a pullback:

\[\begin{tikzcd}[ampersand replacement=\&]
	{G(u)} \& {F(u)} \\
	{G(m)} \& {F(m)}
	\arrow["{\eta_u}", from=1-1, to=1-2]
	\arrow["{\eta_m}", from=2-1, to=2-2]
	\arrow["{G(m\leq u)}"', from=1-1, to=2-1]
	\arrow["{F(m\leq u)}", from=1-2, to=2-2]
\end{tikzcd}\]

\begin{example} \label{dense-ex}
\begin{enumerate}
\item Every isomorphism $\eta : G \overset\cong\to F$ is a dense (mono)morphism.

\item  If a point $\pi : 1 \to F$ is such that $\forall u \in Q  \exists m \in Q, u^-\leq m \leq u, F(m\leq u) : F(u) \to F(m)$ is bijective, then  $\pi : 1 \to F$ is a dense monomorphism. In particular, every point $\pi : 1 \to \Omega^-$ is a dense monomorphism. 

\item Let $a, b \in Q$. If $b \leq a$, let $\eta : Q(-,b) \to Q(-, a)$ be the unique monomorphism (a inclusion, in fact). Then $\eta$ is a dense monomorphism iff   $\forall u \in Q \forall y \in [u,a] \exists m \in Q, u^- \leq m \leq u, (y \in [m, b] \Leftrightarrow y \in [u,b])$; therefore, taking $m = u$, we have that $Q(-,b) \hookrightarrow Q(-, a)$ is a dense inclusion.
\end{enumerate}
\end{example}

We register the following (straightforward) result:
\begin{proposition} \label{dense} A pullback of a dense (mono)morphism in $Sh(Q)$ is a dense (mono)morphism.
\end{proposition}
\begin{theorem} \label{classifier(-)-th} 
Suppose that $Q$ is a (commutative, semicartesian and) geometric quantale. Then the sheaf $\Omega^-$ introduced in Example \ref{classifiercand(-)-ex} essentially  classifies the dense subobject in the category $Sh(Q)$. More precisely:
\begin{enumerate}
    \item  $\top : 1 \to \Omega^-$, given by $\top_u : \{*\} \to \Omega^-(u), \top_u(*) = u^-$ determines a dense monomorphism in $Sh(Q)$.
    \item For each dense monomorphism of sheaves $m: S \rightarrowtail F$, there is a unique  morphism $\chi^m : F \to \Omega^-$, such that $\chi_m^- = \chi_m$, and such the diagram below is a pullback. 
    Moreover, for each morphisms $\phi, \phi' : F \to \Omega^-$  that determine  pullback diagrams, it holds: ${\phi}^- = {\phi'}^-$.
\end{enumerate}

\end{theorem}

  \begin{proof}

\begin{enumerate}

    \item Since $u^-$ is an idempotent, then $u^- \in \Omega^-(u)$ (in fact, $u^- = max \Omega^-(u)$). If $v \leq u$ then, by  Proposition \ref{prop:properties_of_idempotentaprox}.(9) $v^-\odot u^- = v^-$, thus $\top = (\top_u)_{u \in Q}$ is a natural transformation. By Proposition \ref{prop:mono_iff_injective}.4, it is clear that $\top$ is a monomorphism of sheaves. In Example \ref{dense-ex}.(2) we argued that $\top$ is dense.
    
    \item First note that: Since $\top$ is a dense monomorphism, it follows from Proposition \ref{dense} that the pullback of a morphism $\phi : F \to \Omega^-$ through $\top$ must be a {\em dense} monomorphism $m : S \rightarrowtail F$. 
    
    Now, note that it is enough establish the result for dense subsheaves  $i_S : S \hookrightarrow F$. We will split the proof in two parts, but first we will provide some relevant definitions and calculations.

      \begin{center}
        $\langle y,u \rangle := \{v \in \Omega^-(u) : F(v\leq u)(y) \in S(v)\}$;\\
        $u_y := \bigvee\langle y,u \rangle$. 
        \end{center}
(a)  If  $v, w \in \Omega^-(u)$  and $w\leq v$, then $v \in \langle y,u \rangle \Rightarrow w \in \langle y,u \rangle$: $S$ is a subpresheaf of $F$.  In particular: if $v \in \langle y,u \rangle$, then $v^- \in \langle y,u \rangle$, since $( \ )^- : \Omega^- \to \Omega^-$ natural transformation, and $v^- \leq v$. 

(b) If $\{v_i : i \in I\} \subseteq \langle y,u \rangle$, then $\bigvee_i v_i \in \langle y,u \rangle$: since $\Omega^-(u)$ is closed under suprema and $S$ is a subsheaf of $F$.

(c) $u_y \in \langle y,u \rangle$ (by (b)) and $u_y^- \in \langle y,u \rangle$ (by (a)). Thus:
$$u_y = max\langle y,u \rangle \mbox{ and } u_y^- = max(\langle y,u \rangle \cap Idem(Q)).$$
(d) $u_y^- = \bigvee \{ e \in Idem(Q) : \exists v, v\preceq u^-, e = v^-, F(v\leq u)(y) \in S(v)\}$ : since $Q$ is a geometric quantale, we have 
\begin{align*}
    u_y^- &= (\bigvee\{ v \in \Omega^-(u) : F(v\leq u)(y) \in S(v)\})^- \\
    &= \bigvee\{v^- : v \in \Omega^-(u), F(v\leq u)(y) \in S(v)\}.
\end{align*}

    \underline{Candidate and uniqueness:}
    Suppose that $\phi : F \to \Omega^-$ is a natural transformation such that the diagram below is a pullback (where $S$ is dense subsheaf of $F$). 

\[\begin{tikzcd}[ampersand replacement=\&]
	S \& 1 \\
	F \& {\Omega^-}
	\arrow["{i_S}"', hook, from=1-1, to=2-1]
	\arrow["{!_S}", from=1-1, to=1-2]
	\arrow["\top", from=1-2, to=2-2]
	\arrow["\phi"', from=2-1, to=2-2]
\end{tikzcd}\]

Note that if $u^- \leq m \leq u$, then, by naturality, $$\phi_m(F(m\leq u)(y)) = \phi_u(y) \odot m^- = \phi_u(y) \odot u^- = \phi_u(y).$$

{\bf Claim (i):} It holds: $u_y^- \leq \phi_u(y) \leq u_y$. Moreover, if $\phi_u(y) \in Idem(Q)$, then $\phi_u(y) = u_y^-$. 

 Since the diagram is a pullback and limits in $Sh(Q)$ are pointwise, then for each $w \in Q$:
    
    $$x \in S(w) \Leftrightarrow x \in F(w) \ and \ \phi_w(x) = w^-$$
    
 Thus, if $v \leq u$ is such that $F(v\leq u)(y) \in S(v)$, then by naturality: $$v^- = \phi_v(F(v\leq u)(y))= \phi_u(y) \odot v^-.$$

Note that $u_y \leq u^- \leq u$ and $u_y \in$ $ \langle y,u \rangle$, thus $u_y^- = \phi_u(y) \odot u_y^-$ and  $u_y^- \leq \phi_u(y)$.

By naturality: $\phi_{\phi_u(y)}(F(\phi_u(y),u)(y))) = \phi_u(y) \odot \phi_u(y)^- = \phi_u(y)^-$

$\phi_u(y) \in \langle y,u \rangle$: since $\phi_u(y) \in \Omega^-(u)$ and $\phi_{\phi_u(y)}(F(\phi_u(y),u)(y))) = \phi_u(y)^-$ then, by the pullback condition, we have that $F(\phi_u(y),u)(y)) \in S(\phi_u(y))$, thus $\phi_u(y) \in \langle y,u \rangle$.

$\phi_u(y) \leq u_y$: since $\phi_u(u) \in \langle y,u \rangle$ and $u_y =  max\langle y,u \rangle$.

If $\phi_u(y) \in Idem(Q)$, then $\phi_u(y) = u_y^-$: since we have established above that $\phi_u(y) \in \langle y,u \rangle$, 
$u_y^- \leq \phi_u(y) \leq u_y$ and because $u_y^- = max(\langle y,u \rangle \cap Idem(Q))$.

Thus, if $\phi_u(y) \in Idem(Q)$, then 
$\phi_{\phi_u(y)}(F(\phi_u(y),u)(y))) 
 = \phi_u(y) = u_y^-$. 
    
{\bf Claim (ii):} If $\phi : F \to \Omega^-$ determines a pullback diagram, then $\phi^- = ( \ )^- \circ \phi$ still determines a pullback.

Since $x \in S(w) $ iff ($x \in F(w)$ and $\phi_w(x) = w^- = (w^-)^- = \phi_w^-(x)$). 

Combining Claim (ii) and Claim (i), $\phi_u^-(y) = u_y^-$ for each $u \in Q$ and $y \in F(u)$, establishing  the required uniqueness assertions. 
    
    \underline{Existence:} For each $u \in Q$ and $y \in F(u)$, define  $\chi^S _u(y) := u^-_y$. Then $(\chi^S_u)_{u \in Q}$ is a natural transformation and it determines a pullback diagram.

Firstly, we will verify that $(\chi^S_u)_{u \in Q}$ is a natural transformation.
Let $u, v \in Q$ be such that $v \leq u$ and let $y \in F(u)$. We have to show that:
$$\chi^S_v (F(v\leq u)(y)) = \chi^S_u(y) \odot v^-$$
    This means:
    $$max(Idem(Q) \cap \langle F(v\leq u)(y),v\rangle) = v^-\odot max(Idem(Q) \cap \langle y,u \rangle) $$
On the one hand, note that 
\begin{align*}
     v^-\odot u_y^- &=  v^-\odot max(Idem(Q) \cap \langle y,u \rangle)\\
     &=v^-\odot \bigvee (Idem(Q) \cap \langle y,u \rangle)\\
     &= \bigvee\{ v^- \odot e: e^2 = e = e\odot u^-, F(eu)(y) \in S(e)\}.
\end{align*}
 Denoting $e':=v^- \odot u_y^-$, we have $e'^2 = e' = e' \odot v^-$ and $F(e'v)(F(v\leq u)(y)) \in S(e')$, thus $e' = v^- \odot u_y^- \in Idem(Q) \cap \langle F(v\leq u)(y),v\rangle$ and then $$max(Idem(Q) \cap \langle F(v\leq u)(y),v\rangle) \geq v^-\odot max(Idem(Q) \cap \langle y,u \rangle).$$
On the another hand,  denote $e'' = max(Idem(Q) \cap \langle F(v\leq u)(y),v\rangle)$. Then $e''^2 = e'' = e'' \odot v^-$ and $F(e''v)(F(v\leq u)(y)) \in S(e'')$. Then $e'' \in Idem(Q), e'' \leq v^- \leq u^-$ and $e'' \in  \langle y,u \rangle$. Thus $e'' \in Idem(Q)$ and $e'' \leq v^-, u_y^-$. Then $e'' = e'' \odot e'' \leq  v^- \odot u_y^-$, i.e., $max(Idem(Q) \cap \langle F(v\leq u)(y),v\rangle) \leq v^-\odot max(Idem(Q) \cap \langle y,u \rangle)$.\\
Now we show that holds the pullback condition for each $u \in Q$:  
$$y \in S(u) \ \Leftrightarrow \ (y \in F(u) \ and \ u^- = \chi^S _u(y) = u^-_y).$$ 
On one hand, let $y \in S(u)$, then $y \in F(u)$ and $u^- \in \Omega^-(u)$ is s.t. $F(u^-\leq u)(y) \in S(u^-)$, since $S$ is a subpresheaf of $F$. Then $u^- \in Idem(Q) \cap \langle y,u \rangle$. Thus, by (b), $u^- \leq u_y^-$. On the other hand $u^-_y \in \Omega^-(u)$, thus $u_y^- \leq u^-$. Summing up: $\chi^S _u(y) = u^-_y = u^-$.

Let $y \in F(u)$ be such that $u^- = \chi^S _u(y) = u^-_y$. Then $u^- = max(\langle y,u \rangle \cap Idem(Q))$. Therefore $F(u^-\leq u)(y) \in S(u^-)$ and, since $i_S: S \hookrightarrow F$ is a dense inclusion,  we have $y \in S(u)$.
\end{enumerate}
\end{proof}

Using $(-)^+$ instead of $(-)^-$ we actually obtain a subobject classifier, $\top : \textbf{1} \to \Omega^+$, but this requires extra conditions on $Q$. Thus, similar to the localic case, there is a pair of inverse bijections  $Sh(Q)(F, \Omega^+) \rightleftarrows Sub(F)$ that are natural in $F$. In particular, $Sh(Q)(\textbf{1}, \Omega^+) \cong Sub(\textbf{1}) \cong Q$ explains how the ``internal truth values object" of $Sh(Q)$, $\Omega^+$, encodes the ``external  truth values" of $Sh(Q)$, the quantale $Q \cong Sub(\textbf{1})$.

We construct it as follows: 

Let $Q$ be a unital, commutative, semicartesian and {\em double distributive} quantale. 

(a) For each $u \in Q$, define $$\Omega^+(u) := \{q \in Q : q^+ \odot u =  q \} = \{q \in Q : q^+ \odot u\leq q \mbox{ and } q \leq u\}.$$

(b)  Note that:
\begin{itemize}
    \item $0 = min (\Omega^+(u)); u = max (\Omega^+(u))$; $\Omega^+(u) = \{q \in Q : q^+ \odot u\leq q \mbox{ and } q \leq u\}$.

    If $q = q^+\odot u$, then $q \leq u$ and $q^+\odot u \leq q$. If $q \leq u$ and $q^+\odot u \leq q$, then $q = q^+\odot q \leq q^+\odot u \leq q $.
    \item  $\Omega^+(u)  = \{q' \in Q : \exists e, e\odot e =e, q' = e \odot u\}$.

    If $q \in \Omega^+(u)$, then  $q = q^+ \odot u$ (by definition) and  $q^+ \in Idem(Q)$. If $e\odot e  = e$, then $e\odot u  \leq u$ and $(e \odot u)^+ \odot u \leq (e^+ \odot u^+) \odot u = e \odot (u^+ \odot u) = e \odot u$. Then, by the previous item, $e\odot u \in \Omega^+(u)$. 
    \item If $q \in \Omega^+(u)$, then $q^+ \odot u^+ = q^+$ and $q^+ \in \Omega^+(u^+)$.

    Suppose that $q \in \Omega^+(u)$. By the  previous observation $q = q \odot u^+$ and, by the minimality of $q^+$, we have $q^+\leq u^+$. Since $q^+, u^+$ are idempotents, we obtain $q^+ = q^+ \odot u^+ = {q^+}^+ \odot u^+$, thus $q^+ \in \Omega^+(u^+)$.
\end{itemize}

(c) Given $v \leq u$ in $Q$, we can define a corresponding restriction map by:\begin{align*}
    \Omega^+(u) &\rightarrow \Omega^+(v)\\
     q\,\,\,\, &\mapsto q^+ \odot v
\end{align*}
\begin{itemize}
    \item $q^+ \odot v \in \Omega^+(v)$:
Since,  in general, for every  $e \in Idem(Q)$, $e \odot v \in \Omega^+(v)$.
    \item If $v \leq u$ and  $q \in \Omega^+(u)$, then $ q^+ \odot v \in \{q' \in Q : q' \odot u^+ = q' \mbox{ and } q' \leq u \} \supseteq \Omega^+(u)$, since:\\
$q^+ \odot v \leq q^+ \odot u = q \leq u$ and $q^+ \odot v =  q^+ \odot v \odot v^+ \leq  q^+ \odot v \odot u^+ \leq q^+ \odot v$.

    \item If $v^+ \odot u = v$, then $v \leq u$ and $\Omega^+(v)\subseteq \Omega^+(u)$:\\
Since $v = v^+ \odot u \leq u$ and if $e \in Idem(Q)$, then $e \odot v = e\odot (v^+ \odot u) = (e \odot v^+) \odot u$ and the result follows from a characterization given in (b).
\end{itemize}
(d)  $\Omega^+$ is a presheaf on $Q$.

The definition of $\Omega^+(u)$ gives $ q =  q^+ \odot u$, $\forall u \in  Q$, $\forall q \in \Omega^+(u)$. 
In other words,  $q_{\restriction_u} = q$. Concerning the composition: If $w\leq v \leq u$, on the one hand 
\begin{align*}
    q_{\restriction_w} &= q^+ \odot w \\
    &= q^+  \odot w \odot w^+ & w \odot w^+ = w \\
    &\leq  q^+  \odot v \odot w^+ & w \leq v\\
    &\leq (q^+  \odot v)^+ \odot w^+ & (-)^+ \mbox{ is a supremum} \\
    &= (q_{\restriction_v})_{\restriction_w}
\end{align*}
On the other hand
\begin{align*}
    (q_{\restriction_v})_{\restriction_w} &= (q^+  \odot v)^+ \odot w  \\
    &\leq q^{+^+} \odot v^+ \odot w & (a \odot b)^+ \leq a^+ \odot b^+ \\
    &=  q^+  \odot v^+ \odot w \\
    & \leq q^+ \odot w \\
    &  = q_{\restriction_w}
\end{align*}

(e) $\Omega^+$ is a separated presheaf.

Let $u = \bigvee\limits_{i \in I} u_i$, and $p, q \in \Omega^+(u)$ such that $q_{\restriction_{u_i}} = p_{\restriction_{u_i}} $, $\forall i \in I.$
 
 Then:

\begin{align*}
    q &= q^+ \odot u =  q^+ \odot \bigvee\limits_{i \in I} u_i = \bigvee\limits_{i \in I} q^+ \odot u_i = \bigvee\limits_{i \in I} q_{\restriction_{u_i}} \\
    &= \bigvee\limits_{i \in I} p_{\restriction_{u_i}} = \bigvee\limits_{i \in I} p^+ \odot u_i = p^+ \odot \bigvee\limits_{i \in I} u_i = p^+ \odot u \\
    &= p
\end{align*}

(f) If we assume, in addition, that $Q$ is, a {\em divisible quantale} satisfying the coherence property below: 

$$(coherence) \hspace{1cm} \forall a,b, a', b' \in Q, \ a \leq b, a' \leq b', \ $$ 
$$   a \odot b' = a' \odot b \ \Rightarrow \  a^+ \odot {b'}^+ = {a'}^+ \odot {b}^+ $$\label{coherence}

Then $\Omega^+$ is a sheaf.

Note that: \\
If $q \in \Omega^+(u)$ and $v \leq u$, then $ q^+ \odot v \in \{q' \in Q : q' \odot u^+ = q' \mbox{ and } q' \leq u \} \supseteq \Omega^+(u)$:\\
Since: $q^+ \odot v \leq q^+ \odot u = q \leq u$ and, as $Q$ is divisible and $v \leq u$, $ (q^+ \odot v) \odot u^+ = q^+ \odot (v \odot u^+) = q^+ \odot (v \odot v^+) = q^+ \odot v$,  by Lemma  \ref{idempclosure+ lemma}.14.

Now we will verify that any compatible family can be glued.

Let $u = \bigvee\limits_{i \in I} u_i$ and consider $\{q_i \in \Omega^+(u_i)\}_{i \in I}$ be a  compatible family. 

This means $q_i^+ \odot (u_i\odot u_j) = q^+_j \odot (u_i \odot u_j)$ and, this is equivalent to 
$$(compatibility) \ q_i \odot u_j = q_j \odot u_i, \ \forall i,j \in I$$

Set $q := \bigvee\limits_{i \in I}q_i^+ \odot u$.

We have that  $q \in \Omega^+(u)$ because, by item (b):

$$q = \bigvee\limits_{i \in I}q_i^+ \odot u \leq  u$$

$$q^+ \odot u =  (\bigvee\limits_{i \in I}q_i^+ \odot u)^+ \odot u \leq  
\bigvee\limits_{i \in I}(q_i^+ \odot u)^+ \odot u \leq \bigvee\limits_{i \in I}(q_i^+ \odot u^+) \odot u) =   \bigvee\limits_{i \in I}q_i^+ \odot u = q$$

We have $q_j \leq q_{\restriction u_j}$.
Indeed: 

$q_j = q_j^+ \odot u_j =  (q_j^+ \odot u_j)^+ \odot  u_j \leq   (\bigvee\limits_{i,k \in I}q_i^+ \odot u_k)^+ \odot u_j = (\bigvee\limits_{i \in I}q_i^+ \odot u)^+ \odot u_j = q_{\restriction u_j}$

It remains to prove that $q_{\restriction u_j} \leq q_j$. We will use the extra hypothesis on the quantale $Q$ to obtain this.

\begin{align*}
    q_{\restriction u_j} 
    &= q^+  \odot u_j  \\
    & = (\bigvee\limits_{i \in I}q_i^+ \odot u)^+ \odot u_j   & \mbox{by definion of } q \\
     & = (\bigvee\limits_{i \in I}( q_i^+ \odot u)^+) \odot u_j \  & \mbox{by Lemma }\ref{idempclosure+ lemma}.16 \\
     & \leq (\bigvee\limits_{i \in I} ({q_i^+}^+ \odot u^+)) \odot u_j \  &\mbox{by Lemma } \ref{idempclosure+ lemma}.9 \\
     & = \bigvee\limits_{i \in I} q_i^+ \odot (u^+ \odot u_j) &  \mbox{by Lemma } \ref{idempclosure+ lemma}.7 \\
      & = \bigvee\limits_{i \in I} q_i^+ \odot u_j \  & \mbox{by Lemma } \ref{idempclosure+ lemma}.15 \\
       & = \bigvee\limits_{i \in I} (q_i^+ \odot u_j^+) \odot u_j \  & \\
        & = \bigvee\limits_{i \in I} (q_j^+ \odot u_i^+) \odot u_j \  & by \ (compatibility) \ and \ (coherence) \\
         & = (\bigvee\limits_{i \in I}  u_i^+) \odot (q_j^+ \odot u_j) \  & \\
          & = (u^+) \odot (q_j) \  &\mbox{by Lemma } \ref{idempclosure+ lemma}.16 \\
          & = q_{j}   \\
\end{align*}
where the last equality holds because 
$q_j \geq (u^+) \odot (q_j) \geq (u_j^+) \odot (q_j) = q_j$ by Lemma \ref{idempclosure+ lemma}.15 and item (b).

Therefore, the gluing exists. By Proposition \ref{lema2.1.5}.1, we conclude $\Omega^+ $ is a sheaf.

\begin{remark} \label{+rem}

\begin{enumerate}

\item There are some examples of quantales that satisfy all the conditions in the above construction: the locales, the quantales of ideals of any PID, the quantales $\mathbb{R}_{+} \cup \{\infty\}$, $\mathbb{N} \cup \{\infty\}$, etc.   Moreover, note that this class of quantales is closed under arbitrary products and that the   mapping $Q \mapsto \Omega^+$  preserves the product construction (see Proposition \ref{prop.productand interval_of_sheaves}).

\item
In the localic case, the sheaf $\Omega^+$ coincides with the subobject classifier sheaf, which we denote by $\Omega_0$. 
Note that if $Q$ is a locale, then $\Omega^+(u) = \{q \in Q \, : \, q^+ \odot u = q\} = \{q \in Q \,:\, q \leq u\} = \Omega_{0}(u)$, and $T_u(*) = u$. Thus  $T : 1 \to \Omega^+$ coincides with the subobject classifier in the category of sheaves on locales \cite[Theorem 2.3.2]{borceux1994handbook3}.
Moreover, note that, in general, the restriction
$\Omega^+_{\restriction Idem(Q)} : Idem(Q)^{op} \to Set$ is such that for each $e \in Idem(Q)$ $\Omega^+_{\restriction Idem(Q)}(e) \cap Idem(Q) = \{e' \in Idem(Q): e' \leq e\}= {\Omega_0}(e) $ is the value of the subobject classifier of the localic topos $Sh(Idem(Q))$.
We will readdress  this subject at Theorem \ref{classifier(+)-th}.

\item  If $Q$ is  the quantale that satisfies the property: for each $q \in Q$, if $q >0$, then $q^+ = \top$. Then $\Omega^+(u) = \{0, u\}$, for each $u \in Q$. Note that this condition holds whenever $Q$ is the quantale of ideals of some PID or one of the linear quantales $\mathbb{R}_{+} \cup \{\infty\}$, $\mathbb{N} \cup \{\infty\}$, for instance.

\end{enumerate}
\end{remark}


\begin{remark}
\begin{enumerate} 

\item Let $Q$ be a commutative, semicartesian, unital and double distributive quantal. For each $v, u \in Q$ such that $v\leq u$ are equivalent: (i) $v \preceq (u \to v)$; (ii) $v^+ \leq (u\to v)$; (iii) $v^+ \odot u \leq v$; (iv) $v^+ \odot u = v$; (v) $v \in \Omega^+(u)$; (vi) $\Omega^+(v) \subseteq \Omega^+(u)$. See Remark \ref{remark:def_implication} to remind the meaning of $u \to v.$ 

\item If a commutative quantale $Q$ satisfies, for each $v \leq u$, the condition (i) above, then it will be called {\em strongly divisible}. Note that a quantale satisfying the hypothesis in previous item and that is strongly distributive, then (by item (iv)) it is divisible.

\item The class of quantales that are commutative, semicartesian, unital, double distributive and strongly divisible -- and its  subclass of quantales satisfying also the condition of {\em coherence} (\ref{coherence}) --,
 contains all locales and is closed under arbitrary products and under interval construction of the type $[e,1]$, where $e \in Idem(Q)$. 

\end{enumerate}
    
\end{remark}







\begin{theorem} \label{classifier(+)-th} 
Suppose that $Q$ is a  (commutative, semicartesian), unital, double-distributive, coherent and {\em strongly divisible} quantale.






Then the sheaf $\Omega^+$ 
classifies all the subobjects  the category $Sh(Q)$. More precisely:
\begin{enumerate}
    \item  $\top : 1 \to \Omega^+$, given by $\top_u : \{*\} \to \Omega^+(u), \top_u(*) = u$ determines a monomorphism in $Sh(Q)$.
    \item For each  monomorphism of sheaves $m: S \rightarrowtail F$, there is a {\em unique}  morphism $\chi^m : F \to \Omega^+$,  such and such the diagram below is a pullback.
\[\begin{tikzcd}[ampersand replacement=\&]
	S \& 1 \\
	F \& {\Omega^+}
	\arrow["{!_S}", from=1-1, to=1-2]
	\arrow["m"', from=1-1, to=2-1]
	\arrow["{\chi_m}"', from=2-1, to=2-2]
	\arrow["\top", from=1-2, to=2-2]
\end{tikzcd}\]
  
    In particular, every monomorphism in $Sh(Q)$ is regular and the category $Sh(Q)$ is balanced (see Remark \ref{balanced-re}).
\end{enumerate}

\end{theorem}
\begin{proof}
The strategy of this proof is similar to the corresponding proof of Theorem \ref{classifier(-)-th}, so we will just will provide details in some parts. 
 
  Let $i_S : S \hookrightarrow F$ be a subsheaf. 
  
  For each $u \in Q$ and each $y \in F(u)$ we define:
$\langle u,y \rangle := \{q \in \Omega^+(u): F(q\leq u)(y) \in S(q)\}$ and $u^y := \bigvee\langle u,y \rangle$. Then $u^y = max\langle u,y \rangle$, since  the  set $\langle u,y \rangle$ is closed under suprema (because $\Omega^+(u)$ is closed under suprema and $S$ is a subsheaf of $F$).

\underline{Candidate and uniqueness:}

Suppose that $\phi : F \to \Omega^+$ is a natural transformation such that the diagram
$$S \overset{i_S}\hookrightarrow F \overset{\phi}\to \Omega^+ \overset{\top}\leftarrow 1 \overset{!_S}\leftarrow S$$
is a pullback.

By naturality, note that for each $u, v \in Q$, if $v \leq u$ and $y \in F(u)$, then 
$$\phi_v(F(v\leq u)(y)) = {\phi_u(y)}^+ \odot v \leq {\phi_u(y)}^+ \odot u = {\phi_u(y)}.$$

By the pullback condition, if $v \leq u$ is such that $F(v\leq u)(y) \in S(v)$, then: 
\begin{align*}(*)
    \hspace{1.5cm}  v = \phi_v(F(v\leq u)(y))= \phi_u(y)^+ \odot v
\end{align*}

Moreover, if $v \in \Omega^+(u)$, then:
\begin{align*}(**)
    \hspace{1.5cm} v \leq \phi_u(y) \odot v^+.
\end{align*}

Indeed, \\ $$v = u \odot v^+ = u \odot (\phi_u(y)^+ \odot v)^+ \leq  u \odot ((\phi_u(y)^+)^+ \odot (v)^+ =  u \odot (\phi_u(y)^+) \odot v^+ = \phi_u(y) \odot v^+$$

Since $u^y \in \langle u,y \rangle$, we obtain, by (*),  
$(\phi_u(y))^+ \odot u^y = u^y$ and, in particular, $u^y \leq \phi_u(y)^+$.

 On the other hand, denote $\phi_u(y) = u' \in \Omega^+(u)$, then $u'=\phi_u(y) = \phi_{u'}(F(u'\leq u)(y))$, thus $u' \in \langle y,u \rangle$. Therefore:
 
\begin{align*}(***)
    \hspace{2cm}  \phi_u(y) \leq   u^y.
\end{align*}

By the relation just above and taking $v = u^y$ in (**) we obtain

 $u^y \leq \phi_u(y) \odot {u^y}^+ \leq u^y \odot {u^y}^+ = u^y$, i.e. $u^y \in \Omega^+(\phi_u(y))$. 
 
 Thus:
 \begin{align*}(****) 
     \hspace{2cm} (\phi_u(y))^+ \odot u^y  = u^y =\phi_u(y) \odot {u^y}^+.
 \end{align*}
 
 Summing up: $ \phi_u(y) \leq   u^y \leq (\phi_u(y))^+$. Therefore $(\phi_u(y))^+ =   (u^y)^+$.

Now we prove that the hypothesis of {\em strongly divisible} on $Q$ entails $\phi_u(y) = u^y$: this establishes the   uniqueness.

By (***) and since $\phi_u(y) \leq u^y$ entails $\Omega^+(\phi_u(y)) \subseteq \Omega^+(u^y)$ and $\phi_u(y) \in  \Omega^+(\phi_u(y))$, then $\phi_u(y) = \phi_u(y)^+ \odot u^y$. Thus by (****) $\phi_u(y) = \phi_u(y)^+ \odot u^y = u^y$.

\underline{Existence:}

For each $u \in Q$ and $y \in F(u)$, define
$$\chi^S_u(y) := u^y = max\{q \in \Omega^+(u): F(q\leq u)(y) \in S(q)\} \in \Omega^+(u).$$

Since $u = max \Omega^+(u)$, then $y \in S(u)$ iff $\chi^S_u(y) = u$. Thus for each $u \in Q$, the map $\chi^S_u : F(u) \to \Omega^+(u)$ define a pullback diagram in the category $Set$.

It remains only to check that $(\chi^S_u)_{u \in Q}$ is a natural transformation. Let $v \leq u$. 

We always have
$\chi^S_u(y)^+ \odot v \leq \chi^S_v(F(v\leq u)(y))$, since:\\
$\chi^S_v(F(v\leq u)(y)) = max\{q' \in \Omega^+(v): F(q'\leq v)(F(v\leq u)(y)) \in S(q')\}$ and
$\chi^S_u(y)^+ \odot v \in \{q' \in \Omega^+(v): F(q'\leq v)(F(v\leq u)(y)) \in S(q')\}$. 

Indeed:
$\chi^S_u(y) = u^y \in \{q \in \Omega^+(u): F(q\leq u)(y) \in S(q)\} \subseteq \Omega^+(u)$, thus: \\
(i) $(u^y)^+ \odot v \in \Omega^+(v) $; \\ (ii) $(u^y)^+ \odot v \leq (u^y)^+ \odot u = u^y$, thus $F((u^y)^+ \odot v\leq u)(y) \in S((u^y)^+ \odot v)$. \\
Therefore:  
$\chi^S_u(y)^+ \odot v \in \Omega^+(v)$ and $F((u^y)^+ \odot v \leq v)(F(v\leq u)(y)) \in S((u^y)^+ \odot v)$.

We will use the hypothesis that $Q$ is strongly divisible to obtain obtain\\
$$\chi^S_u(y)^+ \odot v \geq \chi^S_v(F(v\leq u)(y)) \mbox{ and that establishing the naturality of } (\chi^S_u)_{u \in Q}.$$

Let $v' =  \chi^S_v(F(v\leq u)(y))$. Since $F(v'\leq u)(y) \in S(v')$ and $v'\in \Omega^+(v)$, by strongly divisibility of $Q$, $\Omega^+(v) \subseteq \Omega^+(u)$. Thus $v' \in \langle y,u \rangle$ and, therefore, $v' \leq u^y$.
Then $\chi^S_v(F(v\leq u)(y)) \leq \chi^S_u(y)$. Therefore  
$$\chi^S_v(F(v\leq u)(y)) = \chi^S_v(F(v\leq u)(y))^+ \odot v \leq \chi^S_u(y)^+ \odot v,$$
as we wish. 
\end{proof}
     
    Finally, we observe that Corollary 4.5 in \cite{nlab:subobject_classifier} says that if $\mathcal{C}$ is a category with all finite limits with subobject classifier, then the poset $Sub(C)$ of subobjects of $C$, for every  object $C$ in $\mathcal{C}$, meets distribute over any existing join. So, we have found that $\Omega^+$ is the subobject classifier of $Sh(Q)$ and, by Theorem \ref{subobj-th}, $Sub(Q(-,1)) \cong Q$. So, the extra condition we had to impose over $Q$ to obtain that $\Omega^+$ is a subobject classifier may be implying that $Q$ has an underlying localic structure.

\section{Final Remarks and Future Works}\label{sec:future_works}

Note that for each  (finite or infinite) cardinal $\theta$, and for each sheaf $F$, there are ``external operations'': $\bigvee^\theta_F, \bigwedge^\theta_F : Sub(F)^\theta \to Sub(F)$ and $*_F : Sub(F) \times Sub(F) \to Sub(F)$; moreover, these operation are natural in $F$. Thus,  for the subclass of commutative semicartesian quantales in the conditions of Theorem \ref{classifier(+)-th}, an application of Yoneda lemma gives us ``internal operations''(= morphism in $Sh(Q)$) $\bigvee^\theta, \bigwedge^\theta : (\Omega^+)^\theta \to \Omega^+$ and $* : \Omega^+ \times \Omega^+ \to \Omega^+$. Based on the relationship between such internal operations, we  have started a study of  the internal logic of the category $Sh(Q)$: these results will be part of a future work.

Also, we are developing a parallel work regarding change of basis for sheaves on quantales: Suppose $\phi : Q \to Q'$ preserves $\odot$,  suprema and unity (thus $\phi$ is a functor). Then we obtain a ``change of basis" functor $\phi_* : Sh(Q') \to Sh(Q)$

$$(F' \overset{\eta'}\to G') \ \mapsto \ (F'\circ\phi^{op} \overset{\eta'_{\phi}}\to G'\circ \phi^{op}) $$

Since limits in categories of sheaves are coordinatewise, this functor $\phi_*$ clearly preserves limits, thus it is natural to pose the question if it is a right adjoint. The candidate to be the left adjoint is the same one that appears in the localic case and so we want to study under what conditions the left Kan extension  $Ran_{\phi^{op}} : PSh(Q) \to PSh (Q')$ restricts to $Ran_{\phi^{op}} : Sh(Q) \to Sh(Q')$. The technical construction will appear somewhere else but we leave here two motivations to go deeper into this question: (i) Kan extensions are known as the ``best approximation" of a given functor through another given functor. Since the inclusion $Idem(Q) \to Q$ preserves multiplication,  suprema and unity we may argue that $Sh(Idem(Q))$ is the best localic topos associated to $Sh(Q)$ and (ii) study sheaves on rings by studying sheaves on topological spaces and vice-versa through functors between the locale of open subsets of a topological space $X$ and the quantale of ideals of a ring generated by the $X$, as the ring of continuous functions, for instance. We are investigating such relations and applying cohomological techniques aiming to find closer relationships between algebra and geometry. 

Moreover, we are exploring a notion of covering appropriate for monoidal categories that is more general than a Grothendieck pretopology such that the respective category of sheaves -- called \textit{Grothendieck lopos} -- encompass both Grothendieck topos and $Sh(Q)$. This approach leads to the development of an \textit{elementary lopos theory}, a generalization of toposes but with a linear internal logic. In \cite{alvim2023mathscr}, the coauthors of the present work describe a category of $Q$-Sets, which will also be used to guide the future definition of a lopos, and when $Q$ is locale we obtain a well-known equivalence between the category of sheaves on a locale and complete $Q$-Sets. Therefore, this paper is one of many steps in the direction of a broader project.

\section{Acknowledgments}

This study was financed by the Coordenação de
Aperfeiçoamento de Pessoal de Nível Superior - Brasil (CAPES) - Finance Code 001.  Ana Luiza Tenório was funded by CAPES. Grant Number 88882.377949/2019-01. Caio de Andrade Mendes was supported by CAPES Grant Number 88882.377922/2019-01.

\bibliographystyle{abbrv} 
\bibliography{ArXiv_sh_q}

\begin{thebibliography}{10}

\bibitem{adamek1994locally}
J.~Ad{á}mek and J.~Rosick{ý}.
\newblock {\em Locally presentable and accessible categories}, volume 189.
\newblock Cambridge University Press, 1994.

\bibitem{aguilar2008sheaves}
J.~M. Aguilar, M.~V.~R. Sanchez, and A.~Verschoren.
\newblock Sheaves and sheafification on {Q}-sites.
\newblock {\em Indagationes Mathematicae}, 19(4):493--506, 2008.

\bibitem{alvim2023mathscr}
J.~G. Alvim, C.~de~Andrade~Mendes, and H.~L. Mariano.
\newblock Q-sets and friends: Categorical constructions and categorical
  properties, 2023.

\bibitem{barr1969adjunction}
M.~Barr, P.~Berthiaume, B.~Day, J.~Duskin, S.~Feferman, G.~Kelly, S.~Mac~Lane,
  M.~Tierney, R.~Walters, and G.~Kelly.
\newblock Adjunction for enriched categories.
\newblock In {\em Reports of the Midwest Category Seminar III}, pages 166--177.
  Springer, 1969.

\bibitem{borceux1994handbook}
F.~Borceux.
\newblock {\em Handbook of categorical algebra: Basic category theory},
  volume~1 of {\em Encyclopedia of Mathematics and its Applications}.
\newblock Cambridge University Press, 1994.

\bibitem{borceux1994handbook2}
F.~Borceux.
\newblock {\em Handbook of Categorical Algebra: Categories and Structures},
  volume~2 of {\em Encyclopedia of Mathematics and its Applications}.
\newblock Cambridge University Press, Cambridge, 1994.

\bibitem{borceux1994handbook3}
F.~Borceux.
\newblock {\em Handbook of Categorical Algebra: Sheaf Theory}, volume~3 of {\em
  Encyclopedia of Mathematics and its Applications}.
\newblock Cambridge University Press, Cambridge, 1994.

\bibitem{borceux1986quantales}
F.~Borceux and G.~V. den Bossche.
\newblock Quantales and their sheaves.
\newblock {\em Order}, 3(1):61--87, 1986.

\bibitem{day1973note}
B.~Day.
\newblock Note on monoidal localisation.
\newblock {\em Bulletin of the Australian Mathematical Society}, 8(1):1--16,
  1973.

\bibitem{dimca2004sheaves}
A.~Dimca.
\newblock {\em Sheaves in topology}.
\newblock Springer Science \& Business Media, 2004.

\bibitem{fourman1979sheaves}
M.~P. Fourman and D.~S. Scott.
\newblock Sheaves and logic.
\newblock In {\em Applications of sheaves}, pages 302--401. Springer, 1979.

\bibitem{grothendieck1972topos}
A.~Grothendieck.
\newblock Th{\'e}orie des topos et cohomologie Étale des schemas ({SGA4}).
\newblock {\em Lecture Notes in Mathematics}, 269:299--519, 1963.

\bibitem{heymans2012grothendieck}
H.~Heymans and I.~Stubbe.
\newblock Grothendieck quantaloids for allegories of enriched categories.
\newblock {\em Bulletin of the Belgian Mathematical Society-Simon Stevin},
  19(5):859--888, 2012.

\bibitem{kashiwara2013sheaves}
M.~Kashiwara and P.~Schapira.
\newblock {\em Sheaves on Manifolds: With a Short History.{\guillemotleft}Les
  d{\'e}buts de la th{\'e}orie des faisceaux{\guillemotright}. By Christian
  Houzel}, volume 292.
\newblock Springer Science \& Business Media, 2013.

\bibitem{lam2006exercises}
T.-Y. Lam.
\newblock {\em Exercises in classical ring theory}.
\newblock Springer Science \& Business Media, 2006.

\bibitem{maclane1992sheaves}
S.~MacLane and I.~Moerdijk.
\newblock {\em Sheaves in geometry and logic: A first introduction to topos
  theory}.
\newblock Universitext. Springer, New York, 1992.

\bibitem{miraglia1998sheaves}
F.~Miraglia and U.~Solitro.
\newblock Sheaves over right sided idempotent quantales.
\newblock {\em Logic Journal of IGPL}, 6(4):545--600, 1998.

\bibitem{mulvey86quantales}
C.~J. Mulvey.
\newblock \&.
\newblock {\em Suppl. Rend. Circ. Mat. Palermo, Ser. II,}, 12:99--104, 1986.

\bibitem{nlab:quantale}
{nLab authors}.
\newblock Quantale.
\newblock \url{https://ncatlab.org/nlab/show/quantale}, 2022.
\newblock \href{https://ncatlab.org/nlab/show/quantale}{}.

\bibitem{nlab:subobject_classifier}
{nLab authors}.
\newblock subobject classifier.
\newblock \url{https://ncatlab.org/nlab/show/subobject+classifier}, May 2023.
\newblock
  \href{https://ncatlab.org/nlab/revision/subobject+classifier/59}{Revision
  59}.

\bibitem{rege1986neumann}
M.~B. Rege.
\newblock On von {N}eumann regular rings and sf-rings.
\newblock {\em Math. Japonica}, 31(6):927--936, 1986.

\bibitem{resende2012groupoid}
P.~Resende.
\newblock Groupoid sheaves as quantale sheaves.
\newblock {\em Journal of Pure and Applied Algebra}, 216(1):41--70, 2012.

\bibitem{rosenthal_quantales}
K.~I. Rosenthal.
\newblock {\em Quantales and their applications}.
\newblock Longman Scientific {\&} Technical, 1990.

\bibitem{walters1981sheaves}
R.~F. Walters.
\newblock Sheaves and {C}auchy-complete categories.
\newblock {\em Cahiers de topologie et g{\'e}om{\'e}trie diff{\'e}rentielle
  cat{\'e}goriques}, 22(3):283--286, 1981.

\end{thebibliography}

\end{document}